\documentclass[letterpaper,11pt]{article}

\usepackage{amsmath, amsthm,amssymb,bbm, bm,multirow,epsfig,subfigure}
\usepackage{paralist}
\usepackage{mathpazo}
\usepackage{hyperref}
\usepackage{verbatim}
\usepackage{mathrsfs}
\usepackage[affil-it]{authblk}
\hypersetup{colorlinks=true}
\usepackage[top=1in, bottom=1in, left=0.8in, right=0.8in]{geometry}

\usepackage{eucal}

\def\N{{\mathbb N}}
\def\Z{{\mathbb Z}}

\def\C{{\mathbb C}}
\def\Z{{\mathbb Z}}

\def\f1{\mathbf{1}}
\newcommand{\indic}[1]{\mathbb{1}\left(#1\right)}
\newcommand{\indicil}[1]{\mathbb{1}(#1)}

\newcommand{\Ept}[1]{\mathbb{E}\left[#1\right]}
\newcommand{\Eptil}[1]{\mathbb{E}[#1]}
\newcommand{\Prob}[1]{\mathbb{P}\left\{#1\right\}}
\newcommand{\Probil}[1]{\mathbb{P}\{#1\}}

\newcommand{\twopartdef}[4]{	\left\{		\begin{array}{ll}	#1 & \mbox{if } #2 \\[3mm]	#3 & \mbox{if } #4	 \end{array}	\right.}

\def\endsup{L}

\def\supint{[0,\endsup)}

\def\xnbar{\overline{X}^{(N)}}
\def\xn{X^{(N)}}

\def\xni{X^{(N),i}}
\newcommand{\xnX}[1]{X^{(N),#1}}

\def\vX{\mathbf{X}^{(N)}}

\def\xin{Z^{(N)}}
\def\xinbar{\overline{Z}^{(N)}}

\def\En{E^{(N)}}

\def\nunone{S^{(N)}}

\def\nunonebar{\overline{S}^{(N)}}

\def\ageni{a^{(N),i}}

\def\age{a}

\def\j0{-\xn(0)+1}
\def\ki0{-k_{0,i}^{(N)}}

\newcommand{\hc}[0] {[0,\infty)}
\newcommand{\ho}[0] {(0,\infty)}

\def\cbone{\C_{\leq 1}^1\hc}
\def\pdeset{\mathbb{X}}

\def\waitn{W^{(N)}}

\def\residni{b^{(N),i}}
\def\lambdan{\lambda^{(N)}}

\def\Sn{\mathcal{S}^{(N)}}

\title{Mean-field Dynamics of Load-Balancing Networks with General Service Distributions}

\author[1]{Reza Aghajani\thanks{Reza@brown.edu. The first author  was partially supported by NSF grant
  CMMI-1234100.}}
\author[2]{Xingjie Li\thanks{\texttt{xli47@uncc.edu}}}

\author[1]{Kavita Ramanan\thanks{Kavita\_Ramanan@brown.edu. The third author was partially supported by ARO grant W911NF-12-1-0222 and
NSF grant CMMI-1407504.}}

\affil[1]{Division of Applied Mathematics, Brown University, Providence, RI, USA.}
\affil[2]{Department of Mathematics and Statistics, University of North Carolina Charlotte, Charlotte, NC, USA.}
\begin{document}

\theoremstyle{definition}
\newtheorem{assumption}{Assumption}[part]
\newtheorem*{assumptionNon}{Assumption}
\newtheorem{theorem}{Theorem}[section]
\newtheorem{lemma}[theorem]{Lemma}
\newtheorem{proposition}[theorem]{Proposition}
\newtheorem{corollary}[theorem]{Corollary}
\newtheorem{definition}[theorem]{Definition}
\newtheorem{remark}[theorem]{Remark}
\renewcommand{\theassumption}{\Roman{assumption}}

\date{\today}

\maketitle
\begin{abstract}
We introduce a general framework for the mean-field analysis of large-scale load-balancing networks with general service distributions.  Specifically, we consider a parallel server network that consists of N queues and operates under the $SQ(d)$ load balancing policy, wherein jobs have independent and identical service requirements and each incoming job is routed on arrival to the shortest of $d$ queues that are sampled uniformly at random from $N$ queues.    We introduce a novel state representation and,
for a large class of arrival processes, including renewal and time-inhomogeneous Poisson arrivals, and mild assumptions on the service distribution, show that the mean-field limit, as $N \rightarrow \infty$, of the state can be characterized as the unique solution of a sequence of coupled partial integro-differential equations, which we refer to as the hydrodynamic PDE.    We use a numerical scheme to solve the PDE to obtain approximations to the dynamics of large networks and demonstrate the efficacy of these approximations using Monte Carlo simulations.   We also illustrate how the PDE can be used to gain insight into network performance.
\end{abstract}

\section{Introduction}\label{sec_intro}
Load balancing is an effective method to improve the performance and reliability of networks by optimizing resource use. With growth in the use of server farms and computer cluster, large-scale load balancing networks appear in variety of applications, including  internet services such as high-traffic web sites,  high-bandwidth File Transfer Protocol sites, Network News Transfer Protocol (NNTP) servers, Domain Name System (DNS) servers, and databases, as well as  cloud computing and communications systems.

An extensively studied problem is the design and analysis of load-balancing algorithms that aim to improve network  performance. This is particularly challenging for large-scale networks, where it is not feasible to implement classical algorithms like \textit{join-the-shortest-queue}, which  incur high communication overhead and computational cost.   In this context,
randomized algorithms provide an attractive alternative.   A popular algorithm that achieves a better balance between network performance and communication overhead is  the so-called $SQ(d)$ (or ``supermarket'') algorithm. This algorithm was introduced in the case $d=2$ by Vydenskaya et al.\ in \cite{VveDobKar96} in the simple setting of a network comprising $N$ homogeneous parallel servers, each with its own queue, that process a common stream of jobs that must be routed immediately on arrival.  In the $SQ(d)$ algorithm, upon arrival of a job, $d$ queues are sampled independently and uniformly at random, and the job is routed to the shortest queue amongst those sampled. Although, when $d\geq2$ and $\lambda<1$, the stationary distribution of a typical queue is not computable,  the limiting stationary distribution, as $N \rightarrow \infty$, was explicitly computed in \cite{VveDobKar96} and shown to have a double exponential tail decay, in contrast to the exponential decay when $d=1$ (which corresponds to random routing). This dramatic improvement in performance gained by adding just one extra random choice is known as the ``power of two choices'', and has led to substantial interest in this class of randomized load balancing schemes. The extension to general $d > 2$ was studied by Mitzenmacher \cite{Mit01}, and a static balls-and-bins analog was originally studied in Azar et al.\ \cite{AzaBrKarUpf99}.

The analysis of the $SQ(d)$ model in the case of exponential service times is carried out using the so-called ``ODE method''. This proceeds by first representing the  dynamics of the $N$-server network by a Markov process $S^{(N)} =(S_\ell^{(N)}; \ell \ge1)$, where $S_\ell^{(N)}(t)$ represents the number of queues that have $\ell$ or more jobs at time $t$, and then showing that, as $N\to\infty$, the sequence of suitably scaled Markov processes converges weakly (on finite time intervals) to the unique solution of a countable system of coupled $[0, 1]$-valued ordinary differential equations (ODEs). This limit result is obtained by a simple application of Kurtz's theorem (see Theorem 11.2.1 in \cite{KurtzBook}), generalized to countable state spaces. Insight into the equilibrium behavior is obtained by first showing that, under the stability condition $\lambda< 1$, each
Markov process $S^{(N)}$ is ergodic, then characterizing the limit of the sequence of scaled stationary distributions as the unique equilibrium point of this system of ODEs \cite{VveDobKar96}, and finally explicitly identifying this equilibrium point.

However, in most real-world applications, service times are typically not exponentially distributed.  For example, statistical analyses suggest that service times follow a Log-Normal distribution in (medium-scale) call centers \cite{BroEtAl05}, a Gamma distribution in Automatic Teller Machines \cite{Kol84}, and studies of content download in Amazon S3 suggest that download times follow a shifted exponential distribution \cite{CheEtAl14,GuaKoz14}. Moreover, Phase-type distributions are used to approximate more complicated service time distributions \cite{DaiDieGao14,See86}. In the case of general service times, in order to describe the evolution of the system it is not sufficient to keep track of the fraction of queues with $\ell$ jobs at any time. For each job in service, one has also to keep track of its age (the amount of time the job has spent in service) or its residual service time. In the system with $N$ servers, this requires keeping track of $N$ additional random variables, and thus the dimension of the Markovian state representation grows with $N$, which is not convenient for obtaining a limit theorem.

\subsection{Prior Work}

The supermarket model and its various modifications have been extensively studied for the case of exponential service distributions. The path-space evolution of the supermarket model was studied by Graham \cite{Gra00}, the maximum equilibrium queue length was analyzed in \cite{LucMcD06}, and strong approximations were obtained in \cite{LucNor05}. The basic model has also been generalized to incorporate features of relevance in applications such as load migration, load stealing and thresholds (see, e.g., \cite{FarMoaPra05}, \cite{Mit00} and \cite{Kat14}).

Although results on the exponential service time model were obtained almost two decades ago, to the best of our knowledge, there appears to be no prior work that characterizes the transient behavior of the supermarket model with general service time distributions, and until recently, there was almost no work on the equilibrium distribution under the sub-criticality condition $\lambda < 1$. Recent progress on the equilibrium behavior was made in a nice series of papers by Bramson et al.\ \cite{BraLuPra10,BraLuPra12,BraLuPra13}, using the so-called ``cavity method''. In particular, in \cite{BraLuPra13}, in was shown that for the sub-class of power law distributions with exponent $-\beta$ with $\beta>1$, the limiting stationary distribution has a doubly exponential tail if $\beta > d/(d-1)$, an exponential tail if $\beta = d/(d -1)$ and a power law tail if $\beta < d/(d -1)$.  However, it should be noted that the tails of the limiting stationary distribution provide only limited information about the stationary distribution of the $N$-server system because limits do not interchange, that is, the tail of the limiting stationary distribution is not the limit of the tails of the $N$-server stationary distributions.

Moreover, the results in \cite{BraLuPra13} assume that the cumulative arrivals are Poisson and the service distribution has a decreasing hazard rate function.   According to the authors \cite[Page 3]{BraLuPra13}, extending their results beyond these assumptions using their framework appears to be a difficult problem.

\subsection{Our Contributions}

In this paper, we introduce a new framework for the study of load balancing networks with general service time distributions. The work combines stochastic modeling (choosing a suitable state representation), with techniques from probability (convergence results), PDEs (analysis of the limit), numerical analysis (stable schemes for solving the PDE) and simulations (validation) to provide engineering insights into the performance of the network under the SQ(2) load balancing algorithm in the presence of general service times. Our specific contributions include the following:

\begin{enumerate}
  \item  Development of the PDE method for analyzing load-balancing networks with general service distributions, which can be viewed as a generalization of the well known ODE method \cite{VveDobKar96,Mit01} used to study large-scale networks with exponential service distributions.  Our method applies to the large class of service distributions that have a density and finite mean, including Pareto, Log-Normal, Gamma and Phase type distributions. In particular, it does not require the decreasing hazard rate function assumption imposed in \cite{BraLuPra13}.

  \item The PDE can be used to approximate not only the queue length distribution but also other quality of service (QoS) parameters such as the virtual waiting time.

  \item In contrast with previous work, our framework also enables the study of transient behavior  of load-balancing networks with general service distributions,   in time-homogeneous networks as well as networks with time-inhomogeneous (Poisson) arrivals, both of which seem relevant for real-world applications.

  \item The stable numerical scheme that we use to  approximate the PDE provides a computationally efficient alternative to Monte Carlo simulations that could be useful for studying the performance and design of large networks, as illustrated in Sections   \ref{sec_simulation} and \ref{sec_insight}.

  \item
  Although, it is known that networks with heavy-tailed service distributions have worse steady-state performances \cite{BraLuPra13}, using our method, we identified the somewhat surprising phenomenon that they can lead to better performance with respect to some transient QoS parameters, as detailed in Section \ref{sec_insight}.

\end{enumerate}

The rest of the paper is organized as follows. The structure and representation of the load-balancing network are described in Section \ref{sec_model}. The hydrodynamic PDE is presented in Section \ref{sec_PDE}. Section \ref{sec_main} contains the main results of the paper, and their proofs are given in Section \ref{sec_proof}. In Section \ref{sec_simulation}, we present the numerical scheme to solve the PDE and validate the main results using Monte Carlo simulations. Engineering insights gained by the PDE are illustrated in Section \ref{sec_insight}. Finally, conclusions and future works are discussed in Section \ref{sec_discussion}.

\section{A Load Balancing Network}\label{sec_model}

\subsection{Model Description}\label{sec_desc}
Consider a network of $N$ homogeneous  parallel servers, each with its own infinite capacity queue, that processes a common stream of arriving jobs (see Figure ~\ref{Figure:network}) that are routed immediately on arrival according to a load-balancing algorithm. Each server follows a first-in-first-out (FIFO) service policy,  and is non-idling in the sense that it cannot be idle if there is a job waiting in the queue. Hence, if a job is routed to an idle server, it  immediately starts receiving service, and otherwise, it is placed at the end of the queue. We index servers by $i=1,...,N$.

Let $\En$ be the cumulative arrival process, that is, for every $t\geq0$, $\En(t)$ is the number of job arrivals during the interval $[0,t]$. Jobs are indexed by $j\in\Z$, and each job $j$ has a service time requirement of $v_j$. The sequence of service times $\{v_j\}_{j\in\Z}$ is assumed to be i.i.d., independent of the arrival process, and distributed according to a general distribution function $G$. Let $\overline{G}$ represent the complementary service distribution, i.e., $\overline{G} (x) = 1 - G(x)$.

For each server $i$ and at each time $t\geq0$, the number of jobs in server $i$ (including the one receiving service and the ones waiting in the queue) is called the queue length of server $i$ at time $t$, and is denoted by $\xni(t)$. The total number of jobs in system is  denoted by $\xn(t)$, whence $\xn(t)\doteq\sum_{i=1}^N\xni(t)$.
\begin{figure}
\centering
\includegraphics[height=3.5 cm]{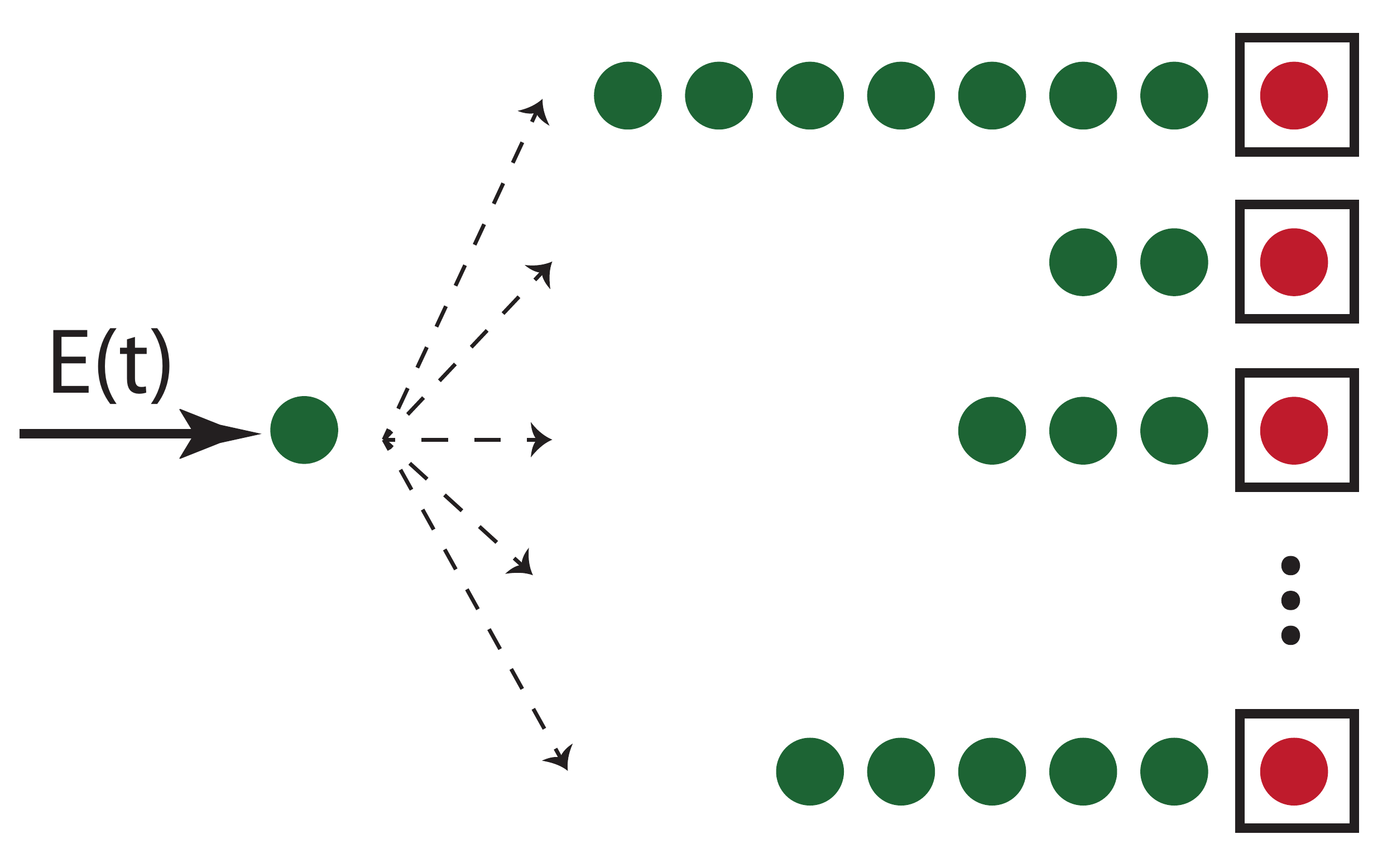}
\caption{The load-balancing network.}
\label{Figure:network}
\end{figure}

We study the performance of the network described above under a randomized load balancing algorithm that we refer to as the $SQ(d)$  algorithm. Based on this algorithm, upon arrival of a job, $d$ queues are sampled independently and uniformly at random, and the job is then routed to the queue with minimum length (with ties broken uniformly at random). In this paper, we present our results for the case $d=2$; however, the extension for
$d \geq 2$ is almost immediate.

\subsection{Assumptions}\label{sec_asm}
The following assumptions are imposed on the cumulative arrival process and service time distribution.

\begin{assumption}\label{asm_arrival}(Arrival Process)
For every $N\in\N,$ $\En$ is  a time-inhomogeneous Poisson process with rate $\lambdan(\cdot)\doteq a^{(N)}\lambda(\cdot)$, where $\lambda(\cdot)$ is a non-negative, locally square integrable function (that is, $\int_0^T\lambda(t)dt<\infty$ for every $T<\infty$) and $\{a^{(N)}\}_{N\in\N}$  is a sequence of positive real numbers satisfying $a^{(N)}/N\to 1$ as $N\to\infty$.
\end{assumption}

\begin{remark}\label{gen_asm_arrival}
Assumption \ref{asm_arrival} is satisfied by most interesting cases including the heavy traffic limit $\lambda^N(\cdot)=N\lambda(\cdot)-\beta\sqrt{N}.$ However, as shown in forthcoming work, the results of this paper hold for more general sequences of arrival rates.
\end{remark}

Recall that when a distribution $G$ has a density $g$, the hazard rate function $h$ associated with that distribution is defined by
\begin{equation}\label{def_h}
  h(x)\doteq\frac{g(x)}{1-G(x)},\quad\quad x\in\supint,
\end{equation}
where $\endsup\doteq\sup\{x\in[0,\infty):G(x)<1\}.$

\begin{assumption}\label{asm_service}(Service Distribution)
   The distribution function $G$ has a finite mean equal to $1$, and density $g$. Also, the hazard rate $h$ is uniformly bounded by a constant $H$, that is,
\[\sup_{x\in\supint}h(x)\leq H.\]
\end{assumption}
\begin{remark}
Assumption \ref{asm_service} is satisfied by a wide class of service time distributions, including any Pareto distribution with finite mean, Phase-Type distributions, Log-Normal distributions and Gamma distributions. Although the boundedness assumption on $h$ can be relaxed, it makes the representation of the results in Sections \ref{sec_main} considerably easier. Note that since for every $a,b \in \supint$,
\[\int_a^bh(x)dx= \ln(\overline G(a))-\ln(\overline G(b)),\]
$h$ is always locally integrable on $\supint$, but not integrable. Therefore, Assumption \ref{asm_service} implies that  $L=\infty$, that is, the distribution $G$ is supported on the whole $[0,\infty)$.
\end{remark}

\subsection{An Important State Descriptor}\label{sec_xi}

For every $\ell\ge 1$ and time $t\geq0$, let $\Sn_\ell(t)$ denote the subset of all servers with queue length at least $\ell$ at time $t$, that is,
\begin{equation}\label{def_setS}
  \Sn_\ell(t)\doteq\left\{i=1,...,N : \xni(t)\geq \ell \right\},
\end{equation}
and denote by $\nunone_\ell$ the number of such servers,  that is,
\begin{equation}\label{def_sl}
  \nunone_\ell(t)\doteq \# \left\{i=1,...,N : \xni(t)\geq \ell \right\}=\# \Sn_\ell(t).
\end{equation}

Note that $\Sn_1(t)$ is the set of servers with at least one job, to which we refer as ``busy servers''. As each server uses a FIFO policy and is non-idling, there is a unique job being served at each busy server. Hence, at each time $t$, we can assign to every busy server $i$ an \textit{age}, denoted by $\ageni(t)$, which is defined to be the amount of time that the current job has been receiving service.

As discussed in the introduction, if the service times are exponentially distributed, the process $\{(\nunone_\ell(t);\ell\ge1);t\geq0\}$ is a  convenient state descriptor for the $N$-server network. But, since other service distributions do not have the memoryless property (exponential is unique), the dynamics of the network depends on  the ages of all jobs in service. Hence, in the case of general serviced distributions, the following state descriptor turns out to be convenient. For every $N\in\N$, $\ell\geq 1$ and $t\geq0$, define the functions $\xin_\ell(t,\cdot)$ as
\begin{equation}\label{def_xin}
  \xin_\ell(t,r)\doteq\sum_{i\in\Sn_\ell} \frac{\overline G(\ageni(t)+r)}{\overline G(\ageni(t))},\quad\quad\quad r\geq0,
\end{equation}
where note that the summation is over servers at queues with length at least $\ell$. For an intuitive understanding of this quantity, note that for any $N\in\N,$ the conditional probability that a job which is being served at time $t$ in a busy server $i$ will still be in system at time $t+r$ for $r\geq0$, is ${\overline G(\ageni(t)+r)}/{\overline G(\ageni(t))}$. Therefore, $\xin_\ell(t,r)$ is the expected number of jobs that are being served at time $t$ in a server with queue length at least $\ell$  and that will still be in the system at time $t+r$.   Note in particular that $\xin_\ell(t,0)$ is equal to the cardinality of $\Sn_\ell(t)$, that is
\begin{equation}\label{ZSrelation}
  \xin_\ell(t,0)=\nunone_\ell(t).
\end{equation}

\section{A PDE Approximation}\label{sec_PDE}

In this section, we introduce a set of partial integro-differential equations that we call the \text{hydrodynamic PDE}. As shown in Section \ref{sec_main},  these coupled equations uniquely characterize the limit, as $N \rightarrow \infty$, of the scaled state descriptors $\xinbar_\ell$ defined as
\[\xinbar_\ell(t,r)\doteq\frac{1}{N}\xin_\ell(t,r).\]
The equations are described in Section \ref{sec_PDEdesc}. Then, in Section \ref{sec_ODE}, we show how these set of PDEs reduce to the set of ODEs obtained in \cite{VveDobKar96} when the service distribution is exponential. Finally, we provide an intuitive interpretation of the hydrodynamic PDE in Section \ref{sec_PDEinterpret}.

\subsection{Description of the Hydrodynamic PDE}\label{sec_PDEdesc}

Define $\cbone$ to be the space of continuously differentiable functions $f$ with bounded derivative, such that $|f|$ is bounded by $1$. Also, define $\pdeset$  to be the set of functions $\varphi$ on $\hc\times\hc$ such that for every $t\geq0$, $\varphi(t,\cdot)\in\cbone$ with $\partial_r \varphi(t,\cdot)$ denoting the derivative with respect to the second variable, and for every $r\geq0$, the mapping $t\mapsto \varphi(t,r)$ is  measurable.  Now we define the  PDE.

\begin{definition}\label{def_pde}
For every function $\lambda$ and set of functions $Z^0_\ell\in\cbone$, $\ell\ge1$, a set of functions $(Z_\ell;\ell\ge1)$ with $Z_\ell\in\pdeset$ for all $\ell\geq 1$ is said to solve the hydrodynamic PDE associated to $(\lambda, Z^0_\ell;\ell\ge1)$ if for every $t,r\geq0$, $Z_1$ satisfies
\begin{align}\label{pde_eq1}
  Z_1(t,r) = Z_1^0(t+r)-\int_0^t \overline G(t+r-u)\partial_rZ_2(u,0)du  + \int_0^t \lambda(u)\overline G(t+r-u)(1-Z_1(u,0)^2)du,
\end{align}
with the boundary condition
\begin{align}\label{pde_boundary1}
Z_1(t,0) =Z_1^0(0)+\int_0^t \big(\partial_rZ_1(u,0)-\partial_rZ_2(u,0)\big)du+ \int_0^t \lambda(u)\left( 1-Z_1(u,0)^2 \right)du,
\end{align}
and for $\ell\geq 2$, $Z_\ell$ satisfies
\begin{align}\label{pde_eq}
  Z_\ell(t,r)  & = Z^0_\ell(t+r)-  \int_0^t \overline G(t+r-u)\partial_rZ_{\ell+1}(u,0)du\\
   &\quad +  \int_0^t \lambda(u)(Z_{\ell-1}(u,0)+Z_\ell(u,0))(Z_{\ell-1}(u,t+r-u)-Z_\ell(u,t+r-u))du,\notag
\end{align}
with the boundary conditions
\begin{align}\label{pde_boundaryell}
Z_\ell(t,0)  =Z_\ell^0(0)+\int_0^t \big(\partial_rZ_\ell(u,0)-\partial_rZ_{\ell+1}(u,0)\big)du  + \int_0^t \lambda(u)\left( Z_{\ell-1}(u,0)^2-Z_\ell(u,0)^2 \right)du.
\end{align}
\end{definition}
Note that by equations \eqref{pde_eq1} and \eqref{pde_eq}, $Z_\ell^0$ is the initial condition for $Z_\ell$, that is $Z_\ell(0,r)=Z_\ell^0(r),$ for all $\ell\geq1$ and $r\geq0.$
\begin{remark}
  Although the equations \eqref{pde_eq1}-\eqref{pde_boundaryell} are partial integro-differential equations and not partial differential equations in the classic sense, we still refer to them as hydrodynamic PDE for conciseness.
\end{remark}

\subsection{Reduction to ODE in the Exponential Case}\label{sec_ODE}
To better illustrate the hydrodynamic PDE, we show that  when the service distribution is exponential and the arrival rate $\lambda (\cdot) \equiv \lambda$ is constant, it reduces to the set of ordinary differential equations (ODEs) obtained in \cite{VveDobKar96,Mit01}. Substituting the complimentary CDF $\overline G(x)=e^{-x}$ for the exponential distribution in definition \eqref{def_xin} of $\xin_\ell$, we have
\begin{equation*}
  \xin_\ell(t,r)=\sum_{i\in\Sn_\ell} \frac{e^{-\ageni(t)-r}}{e^{-\ageni(t)}}=e^{-r}\nunone_\ell(t).
\end{equation*}
As made rigorous in Theorem \ref{thm_conv}, this suggests the limit satisfies $Z_\ell(t,r)=e^{-r}S_\ell(t)$, with $S_\ell$ equal to the limit of $\nunone_\ell/N$. Substituting $\overline G(r)=e^{-r}$, $Z_\ell$ and its derivative $\partial_r Z_\ell(t,r)=-e^{-r}S_\ell(t)$ in \eqref{pde_eq}, for every $\ell\geq2$ we obtain
\begin{align}\label{temp_ODE}
   S_\ell(t)  =e^{-t}S_\ell(0)+\int_0^t e^{-(t-u)}S_{\ell+1}(u)du  +\lambda\int_0^te^{-(t-u)}(S_{\ell-1}(u)^2-S_\ell(u)^2)du.
\end{align}
Taking derivatives with respect to $t$ of both sides of the equation above yields
\begin{align*}
\frac{d}{dt}S_\ell(t)&  = -e^{-t} S_\ell(0) + S_{\ell+1}(t) - \int_0^t e^{-(t-u)} S_{\ell+1}(u) du   -\lambda \left(S_{\ell-1}(t)^2-S_\ell(t)^2\right) \\
               &\quad + \lambda\int_0^t  e^{-(t-u)} \left(S_{\ell-1}(t)^2-S_\ell(u)^2\right) du.
\end{align*}
Using \eqref{temp_ODE} to eliminate the third and fifth terms on the right-hand side of the last equation , we obtain
\begin{equation}\label{odell}
  \frac{d}{dt}S_\ell(t)=-(S_\ell(t)-S_{\ell+1}(t))+\lambda(S_{\ell-1}(t)^2-S_\ell(t)^2).
\end{equation}
Exactly analogous calculations can be shown to see that equations \eqref{pde_eq1}
\begin{equation}\label{ode1}
  \frac{d}{dt}S_1(t)=-(S_1(t)-S_2(t))+\lambda(1-S_1(t)^2).
\end{equation}
Similarly, substituting $\overline G$ and $Z_\ell$  in \eqref{pde_eq},  we obtain
and, again by taking derivative, we have
Moreover, substituting $Z_\ell(t,0)=S_\ell(t)$ and $\partial_r Z_\ell(t,r)=-S_\ell(t)$ in \eqref{pde_boundary1} and \eqref{pde_boundaryell}, we again obtain \eqref{ode1} and \eqref{odell}, respectively. Note that the ODEs \eqref{ode1}-\eqref{odell} coincide with the equations that were previously obtained, e.g., (1) in \cite{Mit01}.

\subsection{Interpretation of the Hydrodynamic PDE}\label{sec_PDEinterpret}

Existence and uniqueness of solutions to the PDE are established in Theorems \ref{thm_conv} and \ref{thm_uniqueness}.  Here, we first provide a heuristic explanation of each term in the hydrodynamic PDE \eqref{pde_eq1}-\eqref{pde_boundaryell}.  Interpreting $Z_\ell$ as the limit of the sequence $\{\xin_\ell/N\}$  defined in \eqref{def_xin}, note that $Z_\ell (t,r)$ represents the (limit of the) fraction of jobs that were in service at time $t$ at a queue of length at least $\ell$ and that were still in service at time $t+r$.  There are three sources that contribute to $Z_\ell (t,r)$, corresponding to the three terms on the right-hand side of \eqref{pde_eq}.  The first is the expected fraction of jobs that were already in service at time $0$ at a queue of length at least $\ell$ and that will still be in service at time $t+r$, which is given by the first term on the right-hand side  of \eqref{pde_eq}, namely $Z_\ell(0, t+r) = Z_\ell^0(t+r)$.

The second source accounts for any job that entered service at some time $u \in [0,t]$ due to completion of the job ahead of it in the queue,  and that would still not have completed service at time $t+r$ (equivalently, jobs that have a service time greater than  $t+r-u$).  For such a job to be in a queue of length $\ell$ or more at time $t$,  it must have been a queue of length $\ell+1$ just prior to the departure of the job ahead of it.   Now, to estimate the expected rate of entry into service of such jobs, consider a server that is busy serving a job $j$ with age $a(u)$  at time $u$.  Given this information, the probability that the job will depart within the next $\epsilon$ units of time is equal to
\begin{align*}
\Prob{v_j<\age(u)+\epsilon|v_j>\age(u)}    = \frac{G(\age(u)+\epsilon) -G(\age(u))}{1-G(\age(u))} \approx h(\age(u))\epsilon.
\end{align*}
Therefore, the expected departure rate of jobs from servers with queues of length $\ell+1$ or greater (conditional on their ages)  is roughly
\[  \sum_{i \in \Sn_{\ell+1}(u)} h (\ageni(u)),
\]
which can be rewritten as
\[ \sum_{i\in\Sn_{\ell+1}(u)}\frac{g(\ageni(u))}{1-G(\ageni(u))}  = - \;\partial_r \xin_{\ell+1}(u,0).  \]
Thus, $-\partial_rZ_\ell(u,0)$ is  (the limit of) the expected departure rate at time $u$ from servers with a queue of length at least $\ell+1$, which is also the expected entry rate into service into a queue of length $\ell$.   Multiplying this by $\overline{G}(t+r-u)$, the fraction of such jobs that will still be in service at time $t+r$, and integrating over all possible values of $u \in [0,t]$, we obtain the second term on the right-hand side of \eqref{pde_eq}.

The last contribution is due to the routing of new jobs at some time $u \in [0,t]$ to a queue of length $\ell-1$ such that the job in service at that queue at time $u$  is still in service at time $t+r$.   Note that such a queue has length $\ell$ or more at time $t$ and the job in service at that queue in time $u$ is in service both at times $t$ and $t+r$, and hence contributes to $Z_\ell (t,t+r)$.  To compute the number of such jobs, first note that the total arrival rate of jobs at time $u$ is $\lambda (u)$, and the expected fraction of such jobs that get routed to a queue of length $\ell-1$ under the $SQ(2)$ routing algorithm is  computed in Section \ref{sec_routing} and is equal to  $\nunonebar_{\ell-1}(u)^2 - \nunonebar_{\ell}(u)^2$, which by \eqref{ZSrelation} is also equal to $\xinbar_{\ell-1}(u,0)^2 - \xinbar_{\ell}(u,0)^2$ (here,  we use the convention $\nunone_0\equiv 1$.)   Now, the total number of jobs in service at a queue of exactly length $\ell-1$ just prior to time $u$ is $\xinbar_{\ell-1}(u-,0) -\xinbar_{\ell}(u-,0)$ and the number of these that will still be in service at time $t+r$ is  $\xinbar_{\ell-1}(u-,t+r-u) - \xinbar_{\ell}(u-,t+r-u)$, which implies the fraction of such jobs still in service at time $t+r$ is given by the ratio.  Multiplying the ratio by the  previously computed arrival rate of jobs into queues of length $\ell-1$ at time $u$ and integrating over $u \in [0,t]$, we obtain the third term.

The boundary conditions \eqref{pde_boundary1} and \eqref{pde_boundaryell} are basically mass balance equations. Recall that $\xin_\ell(t,0)$ is the number of jobs receiving service at a queue of length at least $\ell$ at time $t$, and similarly, $\xin_\ell(0,0)$ is the number of jobs that were receiving  service at time $0$ at such a queue. Note that $\nunone_\ell=\xin_\ell(\cdot,0)$ decreases only due to departures from servers with queue length equal to $\ell$ and increases only due to arrivals routed to servers with queue length equal to $\ell-1$. As discussed above, $-\partial_r\xin_{\ell+1}(u,r)$ is the departures rate at time $u$ of jobs from a queue of length $\ell+1$ or more at the time of departure, and hence the second term on the right-hand side of \eqref{pde_boundaryell} represents the limit of the cumulative number of such departures in the interval $[0,t]$ from queues of length exactly $\ell$. Finally, the  third term on the right-hand side of \eqref{pde_boundaryell} represents the total number of arrivals to servers at queues of length exactly $\ell-1$ in the interval $[0,t]$, whose form can be deduced from the routing probabilities computed above.

\section{Main Results}\label{sec_main}

In this section we state the main results of the paper; the proofs are deferred to Section \ref{sec_proof}. First in Section \ref{sec_mainUnique}, we show that the  PDE  \eqref{pde_eq1}-\eqref{pde_boundaryell} have at most one solution in a suitable space. Then in Section \ref{sec_mainConv}, we show that under suitable assumptions on the sequence of initial distributions, the scaled sequence $\{\xinbar_\ell\}_{N\in\N}$ has a limit which satisfies the  PDE (and hence, is the unique solution). Using this convergence result, we show in Theorem \ref{thm_ql} and Theorem \ref{thm_wt} that the queue length distribution of a typical queue and the mean virtual waiting time in the $N$-server network converges to certain functionals of  the solution to the hydrodynamic PDE. Moreover, we establish a ``propagation of chaos'' result, showing that, at any finite time, the queue length distributions of any finite set of queues are asymptotically independent.

\subsection{Uniqueness of the Solution to the PDE}\label{sec_mainUnique}
We now state our first result. We denote by $\mathbb{L}^1_{\text{loc}}\ho$ the space of locally integrable functions (that is, functions $f$ on $[0,\infty)$ that satisfy $\int_0^T f(t)dt<\infty$ for every $T<\infty$).
\begin{theorem}\label{thm_uniqueness}
  Suppose Assumptions \ref{asm_arrival}-\ref{asm_service} hold. Then, for every non-negative function $\lambda\in\mathbb{L}^1_{\text{loc}}\ho$ and $Z^0_\ell\in\cbone, \ell\geq 1 $, then the hydrodynamic PDE  \eqref{pde_eq1}-\eqref{pde_boundaryell} has at most one solution.
\end{theorem}

\subsection{Convergence of the State Descriptor ${\xin}$}\label{sec_mainConv}

Theorem \ref{thm_conv} below shows that the solution to the hydrodynamic PDE characterizes the limit of $\xinbar$ as $N$ gets large. To state the theorem, we need to impose an additional assumption on the initial conditions. 
\renewcommand{\theenumi}{\alph{enumi}}
\begin{assumption}\label{asm_initial}(Initial Condition)
  \begin{enumerate}
    \item  Almost surely, for every $\ell\ge1$ and every function $f\in\C_b\hc$, the limit
       \[ \lim_{N\to\infty}\frac{1}{N}\sum_{i\in\Sn_\ell(0)}f(\ageni(0)). \]
       exist. Moreover, $\limsup_N\Eptil{ \xnbar(0) }<\infty$. \label{asm_initial_Xnu}
    \item For every $N$, the initial conditions of the $N$-server network are exchangeable, in the sense that for every permutation $\sigma$ on the set of server indices $\{1,...,N\}$, the distribution of the vector
    \begin{equation*}
        \left( \xni(0),\indicil{\xni(0)>0}\ageni(0);i=1,...,N\right)
    \end{equation*}
    does not depend on the choice of $\sigma$. \label{asm_initial_exchange}
  \end{enumerate}
\end{assumption}
Note that in particular, substituting $f=\overline G(\cdot+r)/\overline G(\cdot)$ in Assumption \ref{asm_initial}.\ref{asm_initial_Xnu} shows that for every $r\geq0$, the following limit exists:
\begin{equation}\label{def_ic}
    Z^0_\ell(r)\doteq \lim_{N\to\infty}\xinbar_\ell(0,r).
\end{equation}
This theorem is established in a companion paper \cite{AghRam15c}.

\begin{theorem}\label{thm_conv}
    Suppose $\lambda$ and $\overline{G}$ satisfy Assumptions \ref{asm_arrival} and \ref{asm_service}, let Assumption \ref{asm_initial}.\ref{asm_initial_Xnu} hold and let $Z_\ell^0, \ell \geq 1,$ be as in \eqref{def_ic}.   Then almost surely, for $\ell \geq 1$ and $t, r \geq 0$, the limit
    \begin{equation}
           Z_\ell(t,r)\doteq\lim_{N\to\infty} \xinbar_\ell(t,r),
    \end{equation}
    exists and is the unique solution of the hydrodynamic PDE \eqref{pde_eq1}-\eqref{pde_boundaryell} associated to     $(\lambda, Z_\ell^0, \ell \geq 1).$   if Assumption \ref{asm_initial}.\ref{asm_initial_exchange} also holds and
    \begin{equation}\label{extra_cond}
          \sum_{\ell\geq1}\sup_{N}\Probil{\xni(0)\geq \ell}<\infty,
    \end{equation}
    then for every $t\geq0$ we have
    \begin{equation}\label{thmextra_sum_eq}
         \sum_{\ell\geq 1}\sup_{N}\Ept{\xinbar_\ell(t,0)}< \infty.
    \end{equation}
\end{theorem}

\subsection{Typical Queue Length Distribution}\label{sec_conv_qlp}

As stated below, Theorem \ref{thm_conv} also allows us to characterize the distribution of a typical queue in the network and establish an asymptotic independence result.

\begin{theorem}\label{thm_ql}
Suppose $\lambda$, $\overline G$, $Z_\ell^0$  are as in Theorem \ref{thm_conv},  and let $(Z_\ell;\ell\ge1)$ be the unique solution to the hydrodynamic PDE associated to $(\lambda,Z_\ell^0;\ell\geq1)$. Then, for every $\ell\geq1$ and $t\geq0,$
\begin{equation}\label{Qconv}
  \lim_{N\to\infty} \Prob{\xnX{1}(t)\geq \ell} = Z_\ell(t,0).
\end{equation}
and the queue length of different servers are asymptotically independent, that is, for every $t\geq0$, $K\geq 1$ and $\ell_1,...,\ell_K\ge1$,
\begin{equation}\label{propOfChaos}
  \lim_{N\to\infty} \Prob{\xnX{1}(t)\geq \ell_1,....,\xnX{K}(t)\geq \ell_K}=\prod_{k=1}^K Z_{\ell_k}(t,0).
\end{equation}
\end{theorem}

\subsection{Convergence of Virtual Waiting Times}\label{sec_conv_wait}

The virtual waiting time at any time $t$ is the time that a virtual customer that hypothetically arrives at $t$ has to wait in order to receive service. As another application of Theorem \ref{thm_conv}, our next result shows that for large $N$, the hydrodynamic PDE also provides an approximation to the mean virtual waiting time in an $N$-server network.

\begin{theorem}\label{thm_wt}
Suppose $\lambda$, $\overline G$, $Z_\ell^0$  are as in Theorem \ref{thm_conv},  and let $(Z_\ell;\ell\ge1)$ be the unique solution to the hydrodynamic PDE associated to $(\lambda,Z_\ell^0;\ell\geq1)$. If, in addition, the initial queue lengths satisfy \eqref{extra_cond} and $\ageni(0)<T_0$ for some $T_0<\infty$ and every $i=1,...,N$, then
\begin{align}\label{Eptconv}
    \lim_{N\to\infty}\Ept{\waitn(t)} & = \sum_{\ell\ge 2}Z_{\ell}(t,0)^2+\sum_{\ell\ge 1} \left[ Z_{\ell}(t,0)+Z_{\ell+1}(t,0)\right] \int_{0}^{\infty} \left[Z_{\ell}(t,r)-Z_{\ell+1}(t,r)\right]dr.
  \end{align}
\end{theorem}

\begin{remark}
The uniform boundedness assumption on the initial ages is reasonable for transient analysis since it will be satisfied by any network that started empty a finite time interval ago.    However, the assumption can be relaxed, as illustrated in Figure ~\ref{Figure:waiting_gamma} in Section \ref{sec_moregeneral}. We impose it only to simplify the proofs.
\end{remark}

\begin{remark}
We can show that in fact, as $N \rightarrow \infty$, the sequence of virtual waiting time distributions (and not just their means) converge to a limit distribution whose characteristic function can be expressed as a functional of the  solution to the hydrodynamic PDE. We do not present the details of the proof, but it is similar to that of Theorem \ref{thm_wt}.
\end{remark}

\section{Proofs of Main Theorems}\label{sec_proof}
Now we prove the main Theorems of the paper stated in Section \ref{sec_main}. The uniqueness result of Theorem \ref{thm_uniqueness} is first proved in Section \ref{sec_pdeUnique}. The proof of Theorem \ref{thm_conv} is rather technical, and requires establishing a limit theorem for a sequence of interactive measure-valued processes describing the $N$-server network, which is carried out in a companion paper \cite{AghRam15c}. Theorem \ref{thm_ql} is proved in Section \ref{sec_proofQL}. Then in section \ref{sec_routing}, we carry out the calculation to compute the so-called routing probabilities corresponding to the $SQ(d)$ algorithm, which is required for the proof of Theorem \ref{thm_wt}. We have singled out this  calculation because it is the only part of the proof which depends on the particular load-balancing algorithm invoked in the network. The proof of Theorem \ref{thm_wt} is then given in Section \ref{sec_proofWT}.

\subsection{Proof of the Uniqueness Theorem}\label{sec_pdeUnique}
Throughout this section, we adopt the following notation. For every function $f$ in $\hc,$ that is bounded on finite intervals and every $T\geq0$, we denote \[\|f\|_T\doteq\sup_{0\leq t\leq T}|f(t)|.\]
Also, for functions $f_1,f_2$ on $\hc$, $f_1*f_2(t)\doteq\int_0^tf_1(s)f_2(t-s)ds$ is the (one-sided) convolution of $f_1$ and $f_2$.

\begin{proof}[Proof of Theorem \ref{thm_uniqueness}]
Fix a non-negative function $\lambda\in\mathbb{L}_\text{loc}^1\hc$ and $Z^0=(Z^0_\ell;\ell\geq1)$, and let $Z=(Z_\ell;\ell\geq1)$ and $\widetilde Z=(\widetilde Z_\ell;\ell\geq1)$ both be solutions to the hydrodynamic PDE \eqref{pde_eq1}-\eqref{pde_boundaryell} associated to $(\lambda,Z^0)$. Defining the functions
  \begin{equation*}
    D_\ell(t)\doteq -\int_0^t\partial_rZ_\ell(s,0)ds,\quad\quad t\geq0,
  \end{equation*}
and using integration by parts, we can rewrite
\begin{align}\label{temp_uniq1}
   \int_0^t \overline G(t+r-s)\partial_r Z_{\ell+1}(s,0)ds  = -\overline G(r)D_{\ell+1}(t)+ \int_0^t D_{\ell+1}(s)g(t+r-s)ds.
\end{align}
Substituting \eqref{temp_uniq1} into \eqref{pde_eq1} and \eqref{pde_eq}, and the definition of $D_\ell$ into \eqref{pde_boundary1} and \eqref{pde_boundaryell}, we see that for every $\ell \geq 1$, $Z_\ell$ satisfies for $t, r \geq 0,$
  \begin{align*}
    Z_\ell(t,r)& =Z^0_\ell(r+t)+\overline G(r)D_{\ell+1}(t) + R_\ell(t,r)  - \int_0^t D_{\ell+1}(s)g(t+r-s)ds
  \end{align*}
with boundary condition
  \begin{equation*}
    Z_\ell(t,0)=Z^0_\ell(0)-D_\ell(t)+D_{\ell+1}(t)+\Lambda_\ell(t),
  \end{equation*}
where
\begin{equation}\label{def_R}
   R_\ell(t)=\twopartdef{\int_0^t\lambda(s)\overline G(t+r-s)\big(1-Z_1^2(s,0)\big) ds}{\ell=1,}{\int_0^t\lambda(s) \big(Z_{\ell-1}(s,0)-Z_\ell(s,0)\big)\big(Z_{\ell-1}(s,t+r-s)-Z_{\ell}(s,t+r-s)\big)ds}{\ell\geq2,}
\end{equation}
and
\begin{equation}\label{def_Lambda}
  \Lambda_\ell(t)\doteq \twopartdef{\int_0^t\lambda(s)\big(1-Z^2_1(s,0)\big)ds}{\ell=1,} {\int_0^t\lambda(s)\big(Z^2_{\ell-1}(s,0)-Z^2_\ell(s,0)\big)ds}{\ell\geq2.}
\end{equation}
Similarly, $\widetilde Z$ satisfies analogous equations. Defining $\Delta H_\ell\doteq H_\ell - \widetilde H_\ell $ for $H=Z,D,R,\Lambda$ and $\ell\geq1$, for all $\ell\geq1$ and $t,r\geq0$ we have
  \begin{align}\label{temp_delta}
    \Delta Z_\ell(t,r)& =\overline G(r)\Delta D_{\ell+1}(t) + \Delta R_\ell(t,r)  - \int_0^t \Delta D_{\ell+1}(s)g(t+r-s)ds,
  \end{align}
with boundary condition
  \begin{equation}\label{temp_bdry}
    \Delta Z_\ell(t,0)=-\Delta D_\ell(t)+\Delta D_{\ell+1}(t)+\Delta \Lambda_\ell(t).
  \end{equation}
Now for every $t\geq0$ and $\ell\geq1$, define
\begin{equation}\label{def_V}
  V_\ell(t)\doteq \sup_{r\geq0}|\Delta Z_\ell(t,r)|.
\end{equation}
Note that $V_\ell$ is measurable because it is the supremum over measurable functions, and is bounded by 2 because  $Z_\ell(t,\cdot)$ and $\widetilde Z_\ell(t,\cdot)$ are both in $\cbone.$ By definition \eqref{def_R} of $R_1$, we have,
\begin{align*}
  \Delta R_1(t,r)& = -\int_0^t \lambda(s)\overline G(t+r-s)\big(Z_1(s,0)+\widetilde Z_1(s,0)\big)\Delta Z_1(s,0)ds,
\end{align*}
and hence,
\begin{equation}\label{bd_r1}
  |\Delta R_1(t,r)|\leq  2\int_0^t \lambda(s)V_1(s)ds,\quad\quad \forall t,r\geq0.
\end{equation}
Similarly for $\ell\geq2$, by definition \eqref{def_R} of $R_\ell,$
\begin{align*}
  \Delta R_\ell(t,r)& = \int_0^t \lambda(s)\big(Z_{\ell-1}(s,0)+Z_\ell(s,0)\big)\big(\Delta Z_{\ell-1}(s,t+r-s)-\Delta Z_\ell(s,t+r-s)\big)ds\\
  &\quad +\int_0^t \lambda(s)\big(\Delta Z_{\ell-1}(s,0)+\Delta Z_\ell(s,0)\big)\big(\widetilde Z_{\ell-1}(s,t+r-s)-\widetilde Z_1(s,t+r-s)\big)ds,
\end{align*}
and hence, for all $t,r\geq0$
\begin{equation}\label{bd_rell}
  |\Delta R_\ell(t,r)|\leq  4\int_0^t \lambda(s)\big(V_{\ell-1}(s)+V_\ell(s)\big)ds.
\end{equation}

Furthermore, to bound $\Delta D_\ell$ for  $\ell\geq2,$ we substitute $\Delta D_{\ell+1}$

from \eqref{temp_bdry} in \eqref{temp_delta} and set $r =0$ to conclude that $\Delta D_\ell$ satisfies the renewal equation
\[\Delta D_{\ell}(t)=g*\Delta D_{\ell}(t)+F_\ell(t),\quad\quad t\geq0,\]
with
\begin{equation*}
  F_\ell(t)\doteq\Delta\Lambda_\ell(t)-g*\Delta\Lambda_\ell(t) + g*\Delta Z_\ell(t,0)-\Delta R_\ell(t,0).
\end{equation*}
Fix $\ell\geq2$ and note that by definition \eqref{def_Lambda} of $\Lambda_\ell$,
%
\begin{align*}
  \Delta\Lambda_\ell(t)& = \int_0^t \lambda(s) \big( Z_{\ell-1}(s,0) +\widetilde Z_{\ell-1}(s,0)\big)\Delta Z_{\ell-1}(s,0)ds  - \int_0^t \lambda(s) \big( Z_{\ell}(s,0) +\widetilde Z_{\ell}(s,0)\big)\Delta Z_{\ell}(s,0)ds,
\end{align*}
and hence,
\begin{equation}\label{bd_Lell}
  |\Delta\Lambda_\ell(t)|\leq 4\int_0^t\lambda(s)\big(V_{\ell-1}(s)+V_\ell(s)\big)ds,\quad\quad \forall t\geq0.
\end{equation}
Recall that $V_\ell(t)\leq 2$ for all $t\geq0$ and $\ell\geq1$, and due to the local integrability of $\lambda,$ \eqref{bd_Lell} implies that $\Delta\Lambda_\ell(t)$ is uniformly bounded on any finite interval $t\in[0,T]$. Hence,
\[g*\Delta\Lambda_\ell(t)\leq \|\Delta\Lambda\|_T\int_0^tg(s)ds<\infty,\quad\quad t\in[0,T].\]
Similarly, the bound \eqref{bd_rell} shows that $\Delta R_\ell(t)$ is also bounded for finite $t.$ Therefore, $F_\ell$ is bounded on finite intervals, and hence by the renewal theorem (see Theorem V.2.4 in \cite{Asm03})
\[\Delta D_\ell(t)  = \int_0^t F_\ell(t-s) dU_G(s),\]
where $U_G$ is the renewal measure corresponding to the service distribution $G$. Since $G$ has a density $g$, $U_G$ satisfies
\[U_G(x)=1+\int_0^xu_G(y)dy,\quad\quad x\geq0,\] and the density $u_G$ is bounded on finite intervals and satisfies the equation $u_G=g*u_G+g$ due to Proposition V.2.7 in \cite{Asm03}. Therefore, $\Delta D_\ell(t)$ can be written as
\begin{align}\label{uniq_D}
  \Delta D_\ell(t)  & = F_\ell(t)+ u_G*F_\ell(t)  =\Delta\Lambda_\ell(t) +  u_G*\Delta Z_\ell(t,0) - \Delta R_\ell(t,0) -u_G*\Delta R_\ell(t,0).
  \end{align}
For every fixed $T\geq0$ and all $0\leq t\leq T$, by the definition of the convolution operator and \eqref{def_V},
\begin{align}\label{bd_xi}
  |u_G*\Delta Z_\ell(t,0)|  &\leq \int_0^t u_G(s) |\Delta Z_\ell(t-s,0)|ds \leq \|u_G\|_T \int_0^t V_\ell(s)ds.
\end{align}
Also, using the bound \eqref{bd_rell} we have
\begin{align}\label{bd_uR}
  |u_G*\Delta R_\ell(t,0)| \leq &4\int_0^t u_G(t-s) \int_0^s \lambda(v) \big(V_{\ell-1}(v)+V_\ell(v)\big)dv\; ds \notag\\
                         = &4\int_0^t \lambda(v) \big(V_{\ell-1}(v)+V_\ell(v)\big) \int_v^{t} u_G(t-s)ds\;dv\notag\\
                         \leq &4U_G(T)\int_0^t \lambda(s) \big(V_{\ell-1}(s)+V_\ell(s)\big)ds.
\end{align}
Bounding the terms on the right-hand side of \eqref{uniq_D} using inequalities \eqref{bd_rell}, \eqref{bd_Lell}, \eqref{bd_xi} and \eqref{bd_uR}, for every $\ell\geq 2$ and $t\leq T$, we have
\begin{equation}\label{bd_D}
  \|\Delta D_{\ell}\|_t \leq C_T \int_0^t (1+\lambda(s))(V_{\ell-1}(s)+V_{\ell}(s))ds,
\end{equation}
with $C_T\doteq 8+4U_G(T)+\|u_G\|_T$.
Next, substituting the bound \eqref{bd_D}, but with $\ell$ replaced by $\ell+1$, \eqref{bd_r1} and \eqref{bd_rell} into \eqref{temp_delta}, we obtain that for all $0\leq t\leq T$ and $r\geq0$,
\begin{equation}\label{bd_Endxi1}
  |\Delta Z_1(t,r)|\leq 4(C_T\overline G(r)+1)  \int_0^t (1+\lambda(s))(V_1(s)+V_2(s))ds,
\end{equation}
and for $\ell\geq2$,
\begin{align}\label{bd_Endxiell}
  |\Delta Z_\ell(t,r)|\leq 4(C_T\overline G(r)+1) \int_0^t (1+\lambda(s))(V_{\ell-1}+V_\ell(s)+V_{\ell+1}(s))ds.
\end{align}
Taking supremum over $r\geq0$ on both sides of \eqref{bd_Endxi1} and \eqref{bd_Endxiell}, we have
\begin{equation}\label{temp_V1}
  V_1(t)\leq 8C_T \int_0^t (1+\lambda(s))(V_1(s)+V_{2}(s))ds,
\end{equation}
and for $\ell\geq2$,
\begin{equation}\label{temp_Vell}
  V_\ell(t)\leq 8C_T  \int_0^t (1+\lambda(s))(V_{\ell-1}(s)+V_\ell(s)+V_{\ell+1}(s))ds.
\end{equation}
Now define
\begin{equation}\label{def_Vtotal}
  V(t)\doteq\sum_{\ell\geq1} 2^{-\ell}V_\ell(s).
\end{equation}
Note that $V$ is measurable, and $V(t)=0$ if and only if $V_\ell(t)=0$ for all $\ell\geq1$. Considering the weighted sums of both sides of \eqref{temp_V1} and \eqref{temp_Vell} over $\ell \geq 2$, we obtain
\begin{equation}
    V(t) \leq   26C_T \int_0^t(1+\lambda(s))V(s)ds,\quad\quad 0\leq t\leq T.
\end{equation}
An application of Gronwall's inequality shows that $V\equiv 0$, and hence, $Z_\ell=\widetilde Z_\ell$ for all $\ell\geq0.$ This completes the proof.
\end{proof}

\subsection{Proof of Theorem \ref{thm_ql}}\label{sec_proofQL}
\begin{proof}[Proof of Theorem \ref{thm_ql}]
First note that since the routing algorithm is symmetric with respect to queue indices, the queue lengths and age distributions, which are initially exchangeable by \ref{asm_initial}.\ref{asm_initial_exchange}, remain exchangeable at all finite times $t\geq0$. In particular, the distribution of the vector $(\xnX{\sigma(i)}(t);i=1,...,N)$ is the same for every permutation $\sigma$ on $\{1,2,...,N\}$. Since $\xin_\ell(t,0)=\nunone_\ell(t)$ and by definition \eqref{def_sl} of $\nunone_\ell(t)$, we have
\begin{align}\label{cortemp3}
  \Ept{\xinbar_\ell(t,0)} = \frac{1}{N}\Ept{\sum_{i=1}^N \indic{\xni(t)\geq\ell}}   =\frac{1}{N}\sum_{i=1}^N \Prob{\xni(t)\geq\ell} =\Prob{\xnX{1}(t)\geq\ell},
\end{align}
where the last equality is due to exchangeability. By  Theorem \ref{thm_conv}, $\xinbar_\ell(t,0)$ converges to $Z_\ell(t,0)$ as $N\to\infty$, almost surely, and since $\xinbar_\ell(t,0) = \nunonebar_\ell(t)$ is bounded by $1$, the convergence also holds in expectation by the bounded convergence theorem. Therefore, \eqref{Qconv} follows on taking the limit as $N\to\infty$ of both sides of \eqref{cortemp3}.

Similarly,  for $m=1, \ldots, n$, since $\xin_{\ell_m}(t,0)=\nunone_{\ell_m}(t)$  by \eqref{ZSrelation} and $\nunone_\ell(t)$ is defined by \eqref{def_sl}, we have
\begin{align}\label{temp_cor2}
  \Ept{\prod_{m=1}^n \xinbar_{\ell_m}(t,0)} & = \frac{1}{N^n}\Ept{\sum_{i_1=1}^N...\sum_{i_n=1}^N \indic{\xnX{i_1}(t)\geq\ell_1}...\indic{\xnX{i_n}(t)\geq\ell_n}}\notag\\
  & = \frac{1}{N^n}\sum_{i_1=1}^N...\sum_{i_n=1}^N \Prob{\xnX{i_1}(t)\geq\ell_1,\xnX{i_n}(t)\geq\ell_n}\notag\\
  & = \Prob{\xnX{1}(t)\geq \ell_1,...,\xnX{n}(t)\geq \ell_n},
\end{align}
where the last equality is again due to exchangeability. By another use of Theorem \ref{thm_conv}, $\prod_{m=1}^n \xin_{\ell_m}(t,0),$ converges to $\prod_{m=1}^n Z_{\ell_m}(t,0)$ as $N\to\infty$, almost surely, and in expectation, using the bounded convergence theorem. Taking the limit as $N\to\infty$ of both sides of \eqref{temp_cor2}, we obtain \eqref{propOfChaos}.
\end{proof}

\subsection{Routing Probabilities}\label{sec_routing}

For every server index $i$ and queue length $\ell$, define the routing probability $p(i,\ell;t)$ to be the conditional probability, given the state of the network at time $t$, that the load-balancing algorithm would route a hypothetical job arriving at time $t$ to server $i$, when its queue length is $\ell$. Defining the vector  $\vX(t)$ of all queue lengths and  ages,
\begin{equation}\label{def_vx}
  \vX(t)\doteq\left(\xni(t),\ageni(t);i=1,...,N\right).
\end{equation}
and denoting  by $\kappa(t)$ the (random) index to which the virtual job arriving at time $t$ is routed,  the routing probability $p(i,\ell;t)$ is defined by
\begin{equation}\label{def_routing}
  p(i,\ell,t)\doteq \indic{\xni(t)=\ell}\Prob{\kappa(t)=i|\vX(t)}.
\end{equation}
Now we compute the routing probabilities $p(i,\ell;t)$ for $\ell\geq 1$ and under the $SQ(2)$ algorithm described in Section \ref{sec_desc}. If $\kappa_1$ and $\kappa_2$ are indices  of queues chosen independently and uniformly at random, then $\kappa(t)$ is the index associated with the shorter queue (with ties being broken uniformly at random). The job is routed to a server of queue length exactly $\ell$ if and only if both $\kappa_1$ and $\kappa_2$ have queue lengths at least $\ell$, and at least one of them has queue length $\ell$. Given the vector of all queue lengths in $\vX(t)$, this happens with probability $(\nunonebar_\ell(t))^2-(\nunonebar_{\ell+1}(t))^2$. Since all servers with queue length equal to $\ell$ are equally likely to be chosen and there are $\nunone_\ell(t)-\nunone_{\ell+1}(t)$ of them, for $\ell\geq 1$ we have
\begin{align}\label{routingell}
   p(i,\ell;t) = \indic{\xni(t)=\ell}\frac{1}{N}\left(\nunonebar_\ell(t)+\nunonebar_{\ell+1}(t)\right)
\end{align}
\begin{remark}
    Note that the form of the routing probability for the $SQ(d)$ algorithm is clearly the same as in the exponential case.    To apply our framework to other load balancing algorithms, one would have to replace the expression in \eqref{routingell} with the routing probabilities associated with that algorithm.
\end{remark}

\subsection{Proof of Theorem \ref{thm_wt}}\label{sec_proofWT}

\begin{proof}[Proof of Theorem \ref{thm_wt}]
Suppose that the virtual job arriving at time $t$ is routed to queue $i$, that is $\kappa(t)=i$. If the server $i$ is idle (i.e., $\xni(t)=0$), the virtual waiting time is zero; otherwise if $\xni(t)=\ell$ for some $\ell\geq$ 1, the virtual waiting time $\waitn(t)$ is the sum of service times $v_j$ of jobs waiting in queue $i$ plus the residual time $\residni(t)$  of the job in service at server $i$ at time $t$, that is,
\[\indic{\kappa(t)=i}\waitn(t)=\sum_{j=1}^{\ell-1}v_j+\residni(t).\]
Summing over all possible queue indices $i$ and queue lengths $\ell$, we have
\begin{equation}\label{W}
  \waitn(t)  = \waitn_1(t)+\waitn_2(t),
\end{equation}
where
\begin{equation}\label{W1}
    \waitn_1(t)\doteq\sum_{i=1}^N\sum_{\ell\geq 1}\indic{\kappa(t)=i,\xni(t)=\ell}\sum_{j=1}^{\ell-1}v_j,
\end{equation}
and
\begin{equation}\label{W2}
   \waitn_2(t)\doteq\sum_{i=1}^N\sum_{\ell\geq 1}\indic{\kappa(t)=i,\xni(t)=\ell}\residni(t).
\end{equation}
To compute the expectation of $ \waitn_1(t),$ note that by Assumption \ref{asm_service}, the service times  $v_j$ of jobs that are still waiting in queues satisfy $\Eptil{v_j}=1$ and are independent of all queue lengths and ages at time $t$, as well as $\kappa(t)$. Therefore, taking the conditional expectation given $\vX(t)$ of both sides of \eqref{W1} and using the definition \eqref{def_routing} of $p(i,\ell,t)$, we have
\begin{align*}
 \Ept{\waitn_1(t)|\vX(t)} & = \sum_{i=1}^N\sum_{\ell\geq1}  \Prob{\kappa(t)=i,\xni(t)=\ell\big|\vX(t)} \sum_{j=1}^{\ell-1}\Ept{v_j}\\
& = \sum_{\ell\geq1}(\ell-1) \sum_{i=1}^N \indic{\xni(t)=\ell}p(i,\ell;t).
\end{align*}
Taking expectations of both sides of the last equation,  substituting $p(i,\ell;t)$ from \eqref{routingell}, recalling that the number of servers with queue length equal to $\ell$ is $\nunone_\ell(t)-\nunone_{\ell+1}(t)$, and using the equality $\nunone_\ell(t)=\xin_\ell(t,0)$ from \eqref{ZSrelation}, we see that
\begin{align}\label{W1final}
\Eptil{\waitn_1(t)} & =\frac{1}{N}\Ept{\sum_{\ell\geq1}(\ell-1) \left(\nunonebar_\ell(t)+\nunonebar_{\ell+1}(t)\right)\sum_{i=1}^N \indic{\xni(t)=\ell}}\notag\\
& =\frac{1}{N}\sum_{\ell\geq1}(\ell-1)\Ept{ \left(\nunonebar_\ell(t)+\nunonebar_{\ell+1}(t)\right)\left(\nunone_\ell(t)-\nunone_{\ell+1}(t)\right)}\notag\\
& =\sum_{\ell\geq1}(\ell-1) \Ept{\left(\xinbar_\ell(t,0)\right)^2-\left(\xinbar_{\ell+1}(t,0)\right)^2}\notag\\
& =\sum_{\ell\geq2}\Ept{\left(\xinbar_\ell(t,0)\right)^2}.
\end{align}

To compute the expectation of $\waitn_2(t)$, note that according to the $SQ(2)$ algorithm, the random queue indices $\kappa_1$ and $\kappa_2$ are independent of all other random variables, and therefore, given $\vX(t)$, $\kappa(t)$ is independent of all residual service times. Therefore, taking conditional expectations of both sides of \eqref{W2} and invoking the definition \eqref{def_routing} of $p(i,\ell;t)$ again, we have
\begin{align}\label{tempW}
& \Ept{\waitn_2(t)|\vX(t)}= \sum_{\ell\geq1} \sum_{i=1}^N \indic{\xni(t)=\ell}p(i,\ell;t) \Ept{\residni(t)|\vX(t)}.
\end{align}
The residual service time $\residni(t)$ of the job that begins being processed by server $i$ at time $t$ satisfies
\[\residni(t)=v_{J(i;t)} -\ageni(t),\]
where $J(i;t)$ is the index of the job begin processes in server $i$ at time $t$. It follows from the i.i.d assumption on the service times and their independence from the arrival process (Assumption \ref{asm_service}) that the residual service time $\residni(t)$ depends on the state variable $\vX(t)$ only through the age $\ageni(t)$. A complete rigorous justification of this intuitive assertion is rather long and technical, and hence will be presented elsewhere. Hence using equation (3.2) of Section V.3 in \cite{Asm03} in the second equality below, for every $r\geq0$, we have
\begin{align}
  \Prob{\residni(t)>r|\vX(t)}& =\Prob{\residni(t)>r|\ageni(t)}=\frac{\overline G(\ageni(t)+r)}{\overline G(\ageni(t))},\notag
\end{align}
and therefore,
\begin{align}\label{bdist}
  \Ept{\residni(t)|\vX(t)} =\int_0^\infty\frac{\overline G(\ageni(t)+r)}{\overline G(\ageni(t))}\;dr.
\end{align}
Now, substituting equations \eqref{routingell} and \eqref{bdist} in \eqref{tempW}, taking expectations of both sides, and using definition \eqref{def_xin} of $\xin_\ell$ and equation \eqref{ZSrelation}, we can see that
\begin{align}\label{W2final}
\Eptil{\waitn_2(t)} & = \sum_{\ell\geq1} \frac{1}{N}\mathbb{E}\left[\left(\nunonebar_\ell(t)+\nunonebar_{\ell+1}(t)\right)
\int_0^\infty \sum_{i=1}^N\indic{\xni(t)=\ell}  \frac{\overline G(\ageni(t)+r)}{\overline G(\ageni(t))} dx\right]\notag\\
&=\sum_{\ell\geq1} \mathbb{E}\left[\left(\xinbar_\ell(t,0)+\xinbar_{\ell+1}(t,0)\right)\int_0^\infty \left(\xinbar_\ell(t,r)-\xinbar_{\ell+1}(t,r) \right)  dr\right].
\end{align}

Finally, taking the limit as $N \rightarrow \infty$ in \eqref{W2final}, we obtain
\begin{equation}\label{temp_w1}
  \lim_{N\to\infty}\Ept{\waitn_1(t)}=\sum_{\ell\geq2}\lim_{N\to\infty}\Ept{\xinbar_\ell(t,0)^2}
  = \sum_{\ell\geq2}\left(Z_\ell(t,0)\right)^2.
\end{equation}
The exchange of limit and summation in the first equality is justified by the bound \eqref{thmextra_sum_eq}, $\xinbar_\ell(t,0)\leq 1$ and the dominated convergence theorem,  while the second equality follows from Theorem \ref{thm_conv} and the bounded convergence theorem. Moreover, by the uniform boundedness assumption on the ages imposed in Theorem \eqref{thm_wt}, $\ageni(t)<T_0+t$. Therefore, since $\int_0^\infty \overline{G}(r)dr = 1$ because the service time has mean $1$, and by definition \eqref{def_xin} of $\xin_\ell$,
\begin{align}\label{integral_bound}
  \int_0^t\xinbar_\ell(t,r)dr & \leq \frac{1}{N}\int_0^t \sum_{i\in\Sn_\ell(t)}\frac{\overline G(r)}{\overline G(T_0+t)}dr\leq \frac{1}{\overline G(T_0+t)}\nunonebar_\ell(t)\int_0^t\overline G(r)dr \leq \frac{1}{\overline G(T_0+t)}.
\end{align}
Therefore, taking the limit as $N\to\infty$ of both sides of \eqref{W2final}, we have
\begin{align}\label{temp_w2}
  \lim_{N\to\infty}\Ept{\waitn_2(t)} & =\sum_{\ell\geq1}\lim_{N\to\infty} \mathbb{E}\left[\left(\xinbar_\ell(t,0)+\xinbar_{\ell+1}(t,0)\right)\int_0^\infty \left(\xinbar_\ell(t,r)-\xinbar_{\ell+1}(t,r) \right)  dr\right]\notag\\
  & =\sum_{\ell\geq1} \mathbb{E}\left[\left(Z_\ell(t,0)+Z_{\ell+1}(t,0)\right)\int_0^\infty \left(Z_\ell(t,r)-Z_{\ell+1}(t,r) \right)  dr\right].
\end{align}
The exchange of limit and summation in the first equality is justified by \eqref{thmextra_sum_eq}, $\xinbar_\ell(t,r)\leq \xinbar_\ell(t,0)\leq 1$ and the dominated convergence theorem,  while the second equality holds due to Theorem \ref{thm_conv}, \eqref{integral_bound} and the bounded convergence theorem. The result follows from \eqref{temp_w1} and \eqref{temp_w2}.
\end{proof}


In this section we validate the results in Theorem \ref{thm_ql} and Theorem \ref{thm_wt} using Monte Carlo (MC) simulations. First in Section \ref{sec_simApprox}, we present a stable scheme for numerically  approximating the solution to the hydrodynamic PDE. Then in Section \ref{sec_simValid}, we compare the results obtained from Monte Carlo simulation and numerical solutions to PDE for a variety of different service distributions and in each case, we observe a close match.

\subsection{A Numerical Approximation Scheme}\label{sec_simApprox}

Note that the hydrodynamic PDE is comprised of a countable number of integro-differential equations, corresponding to each $Z_\ell$. Therefore, in order to numerically solve the hydrodynamic PDE \eqref{pde_eq1}-\eqref{pde_boundaryell}, we truncate the state space and only solve for $Z_\ell (t,r)$, $\ell  = 1, \ldots, L_0, 0 \leq t \leq T_0, 0 \leq r \leq R_0,$ for suitable $L_0  \in \mathbb{N}$ and  $R_0, T_0 < \infty$. Next, we discretize $r$ and $t$ on uniform meshes: $\hat Z_{\ell}(t_m,r_n)\doteq Z_{\ell}(m\delta, n\delta)$ with $0\le m\le \lfloor T_0/\delta\rfloor$ and $0\le n\le \lfloor R_0/\delta\rfloor$. Then we solve the equations \eqref{pde_eq1}-\eqref{pde_boundaryell} numerically by the finite difference method and the explicit forward Euler scheme\cite{Gro07}.  That is, for fixed $t_m>0$ and $r_n>0$, we apply the following approximations
\begin{align*}
\begin{split}
\hat Z_{1}(t_m+\delta,r_n)& =\hat Z_{1}(t_m, r_n+\delta)  - \int_{t_m}^{t_m+\delta}\overline{G}(t_m+\delta+r_n-u)\partial_{r}\hat Z_{2}(u,0)du\\
&\quad+\int_{t}^{t+\delta}\lambda(u)\overline{G}(t_m+\delta+r_n-u)\left(1-\hat Z_1(u,0)^2\right)du\\
&\approx  \hat Z_{1}(t_m,r_n+\delta)-\overline{G}(r_n)\left(\hat Z_{2}(t_m,\delta)-\hat Z_2(t_m,0)\right) +\overline{G}(r_n)\left(1-\hat Z_{1}(t_m,0)^2\right)\int_{t}^{t+\delta} \lambda(u)du.
\end{split}
\end{align*}
and for $\ell\ge 2$,
\begin{align*}
\hat Z_{\ell}(t_m+\delta,r_n)& = \hat Z_{\ell}(t_m, r_n+\delta) -\int_{t_m}^{t_m+\delta}\overline{G}(t_m+\delta+r_n-u)\partial_{r}\hat Z_{\ell+1}(u,0)du\\
&\quad+\int_{t_m}^{t_m+\delta}\lambda(u)\left(\hat Z_{\ell-1}(u,0)+\hat Z_{\ell}(u,0)\right)\left(\hat Z_{\ell-1}(u,t_m+\delta+r_n-u)-\hat Z_{\ell}(u,t_m+\delta+r_n-u)\right)du\\
&\approx\hat Z_{\ell}(t_m,r_n+\delta)-\overline{G}(r_n)\left(\hat Z_{\ell+1}(t_m,\delta)-\hat Z_{\ell+1}(t_m,0)\right)\\
&\quad +\left(\hat Z_{\ell-1}(t_m,0)+\hat Z_{\ell}(t_m,0)\right)\left(\hat Z_{\ell-1}(t_m,r_n)-\hat Z_{\ell}(t_m, r_n)\right)\int_{t_m}^{t_m+\delta}\lambda(u)du.
\end{align*}
The numerical scheme above has been stabilized by applying the upwind scheme \cite{CouEtAl52} to the derivatives with respect to $r$, otherwise, the numerics will blow up in a short time interval. Using these approximations, we can update $\hat Z_{\ell}\left( t_m+\delta, \cdot\right)$ from $\hat Z_{\ell}(t_m,\cdot)$.

\subsection{Validation of Results}\label{sec_simValid}

We now compare the empirical queue length distribution and mean virtual waiting time obtained from the Monte Carlo  simulation of an $N$-server network, with the corresponding limit quantities \eqref{Qconv} and \eqref{Eptconv}, as predicted by the numerical approximation of the PDE.

We consider a sequence of networks indexed by the number of servers $N$, with a Poisson arrival process with rate $\lambda=0.5$, and the following initial conditions. Each server has initially one job with initial age equal to zero, that is  $\xni(0)=1$ and $\ageni(0)=0$ for all $i=1,...,N.$ Note that this sequence of initial conditions satisfies the conditions of Theorem \ref{thm_conv}, and converges to the initial condition $(Z_\ell^0(\cdot);\ell\geq1)$ for the hydrodynamic PDE, where for $r\geq0$,
\[Z^0_1(r)=\overline G(r),\quad\quad\quad Z^0_\ell(r)=0,\quad \ell\geq2.\]

\subsubsection{Queue Length Distribution}
\section{Simulation Results}\label{sec_simulation}
\begin{figure*}
\begin{center}
\subfigure{\includegraphics[height = 2.8 cm ]{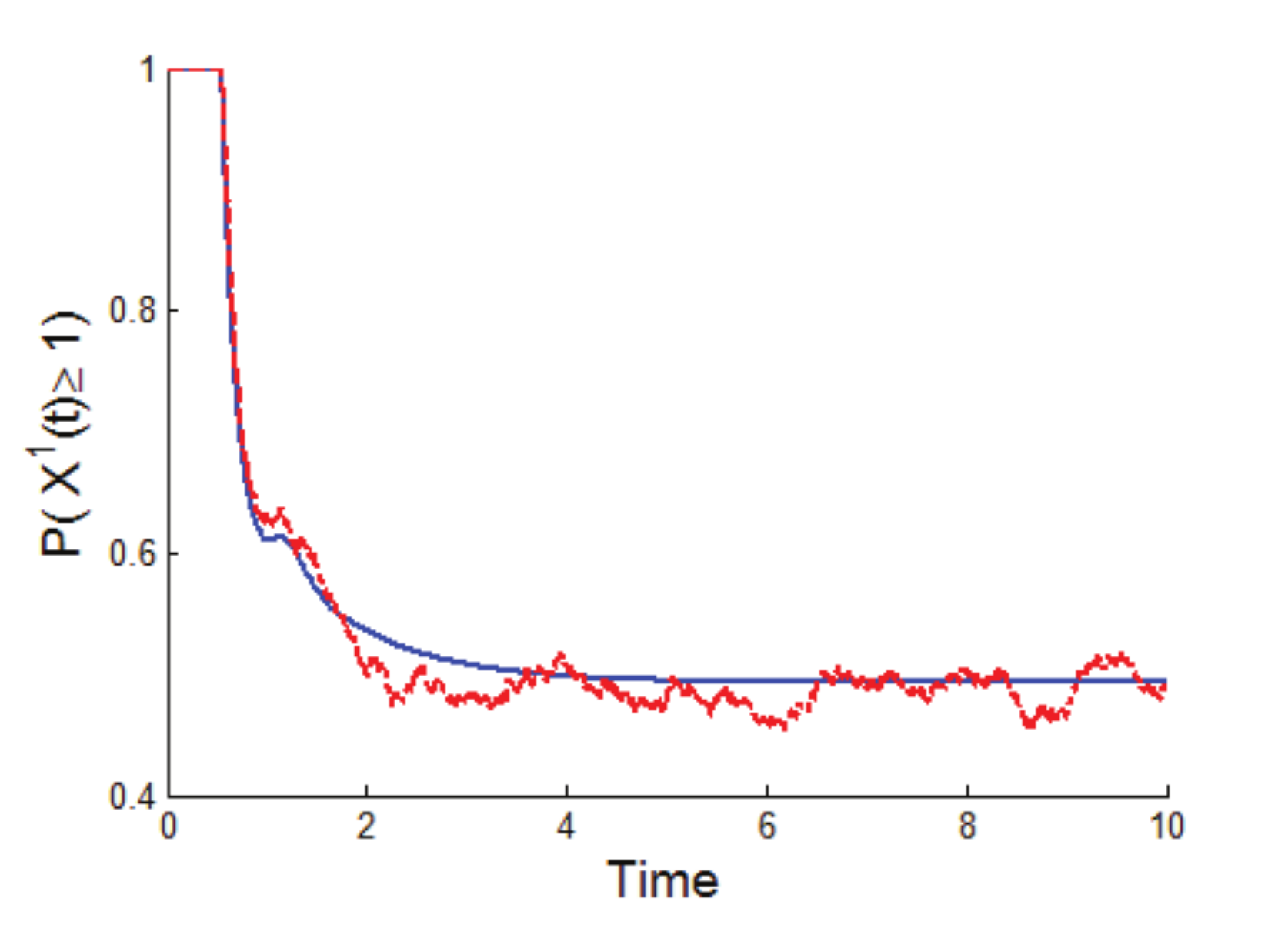}}\quad
\subfigure{\includegraphics [height =2.8 cm]{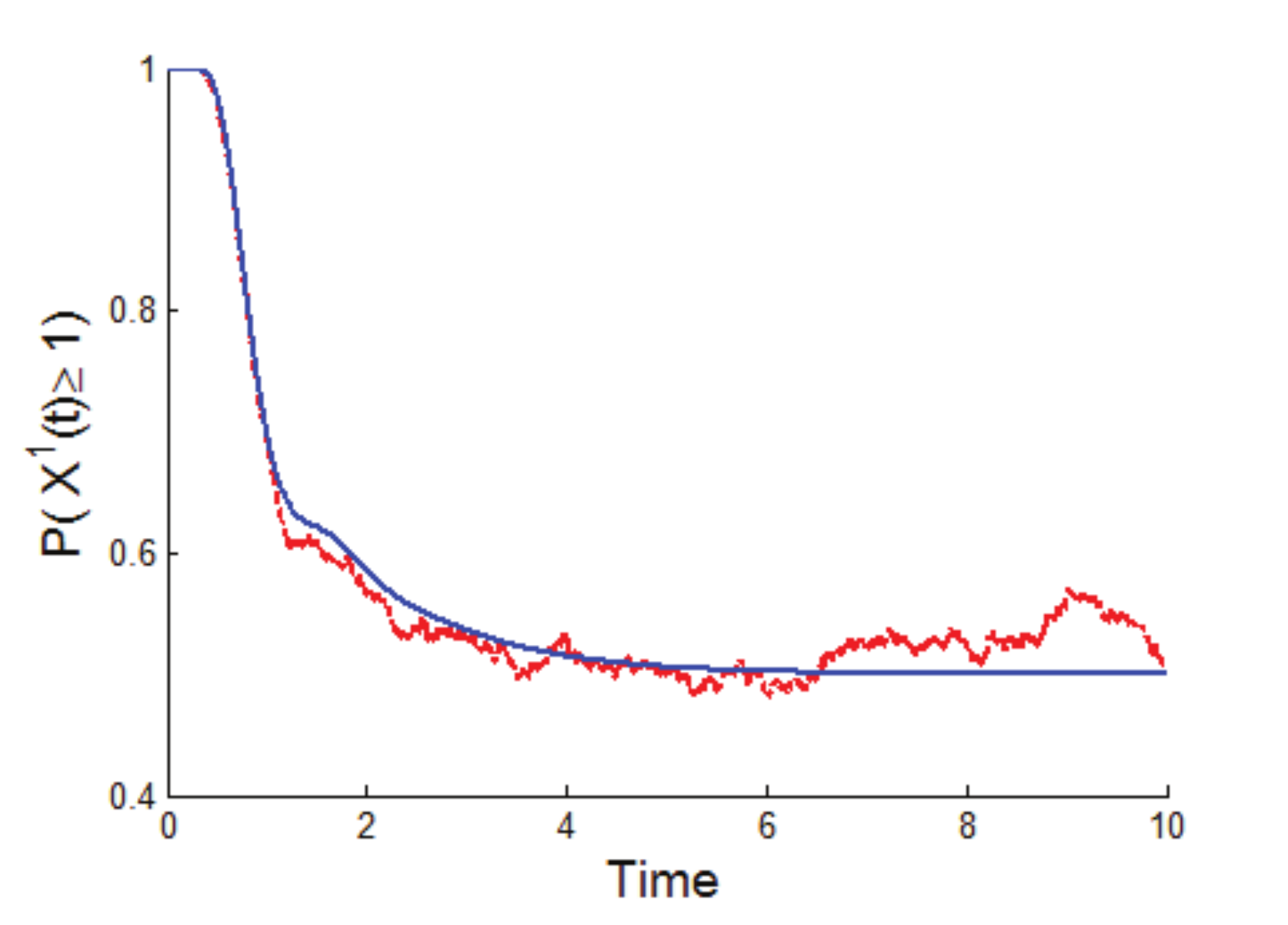}}\quad
\subfigure{\includegraphics[height = 2.8 cm ]{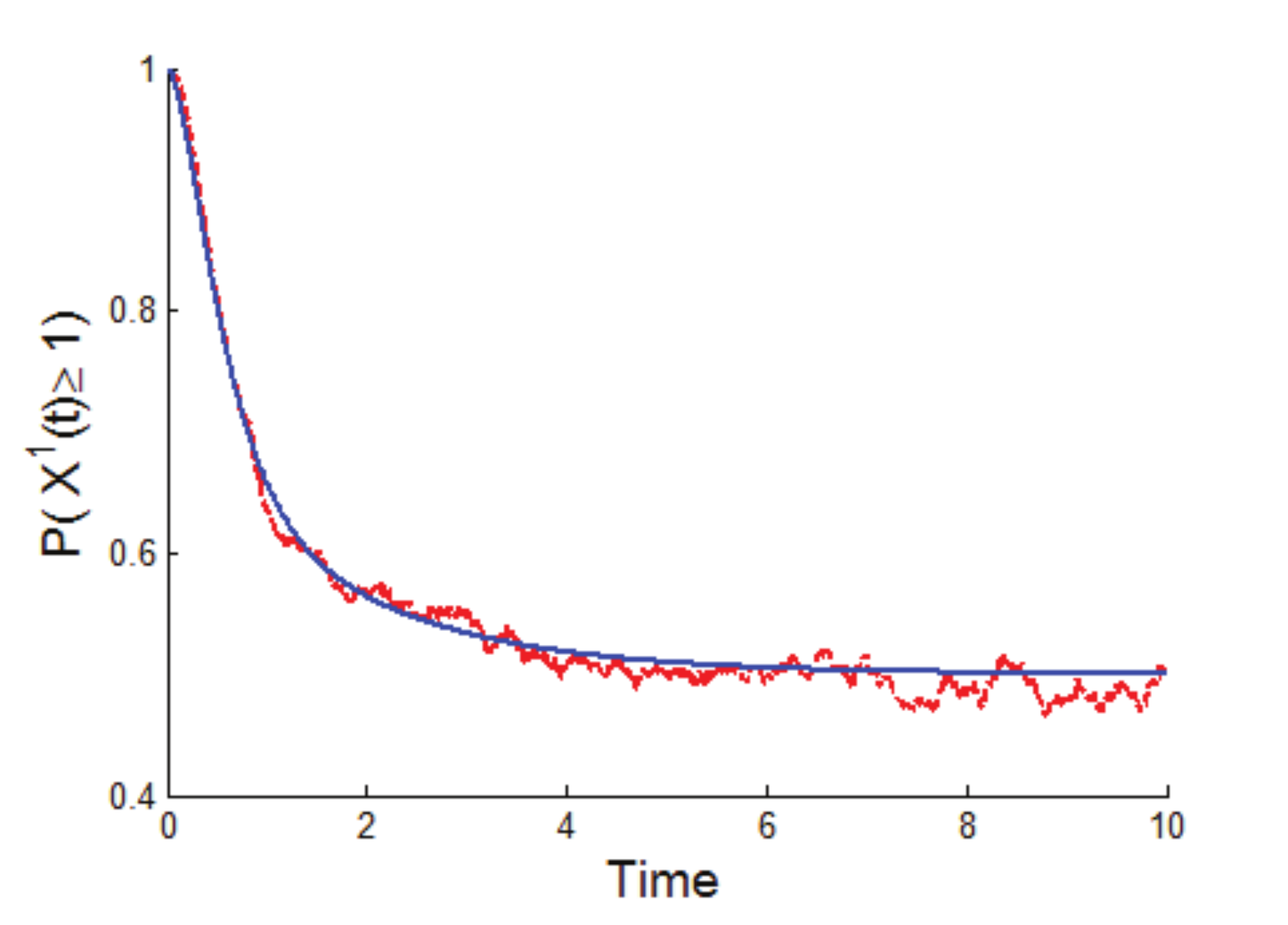}}\quad
\subfigure{\includegraphics[height = 2.8 cm ]{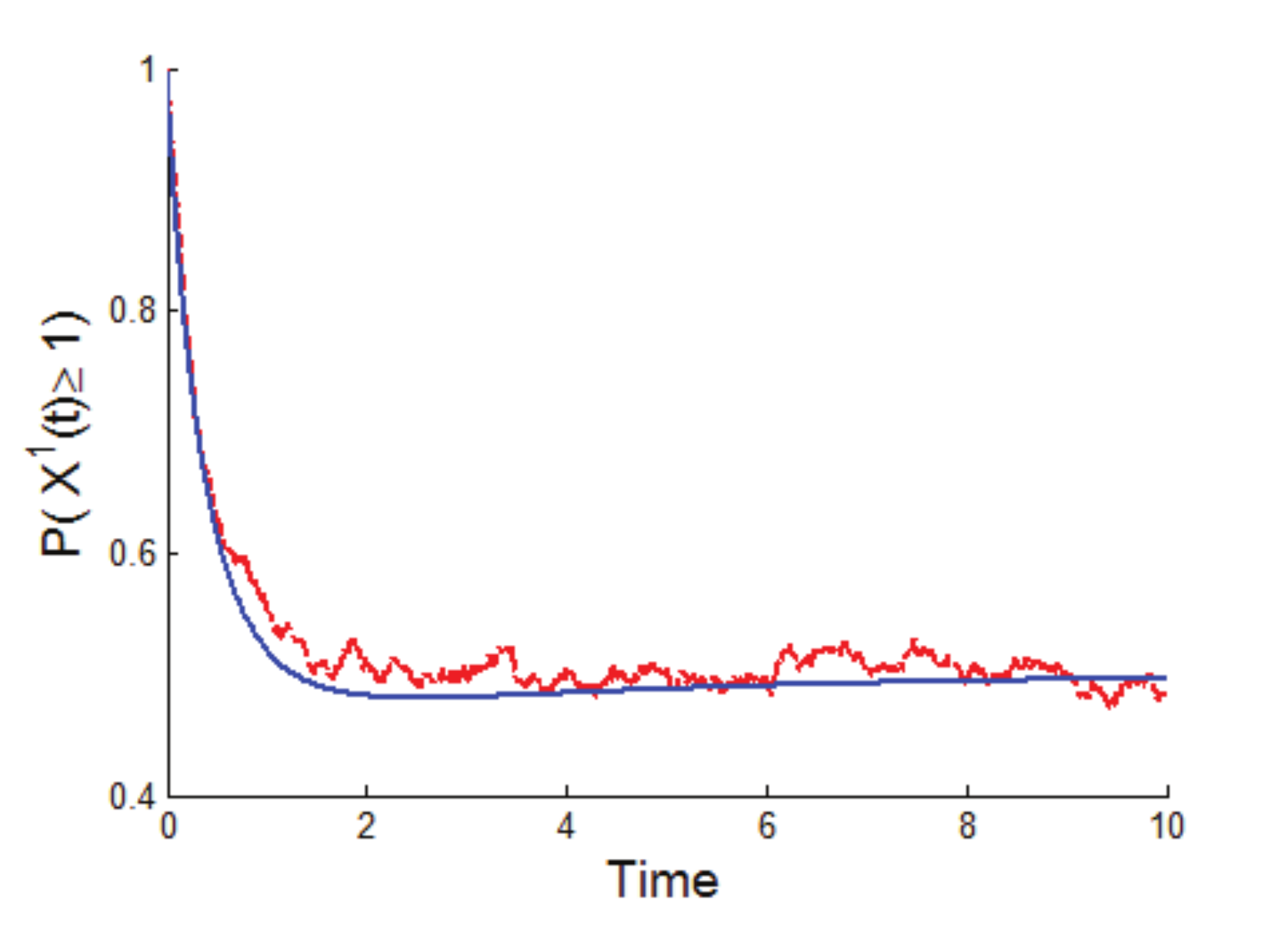}}
\end{center}
\begin{center}
\addtocounter{subfigure}{-4}
\subfigure[Pareto]{\includegraphics[height = 2.8 cm ]{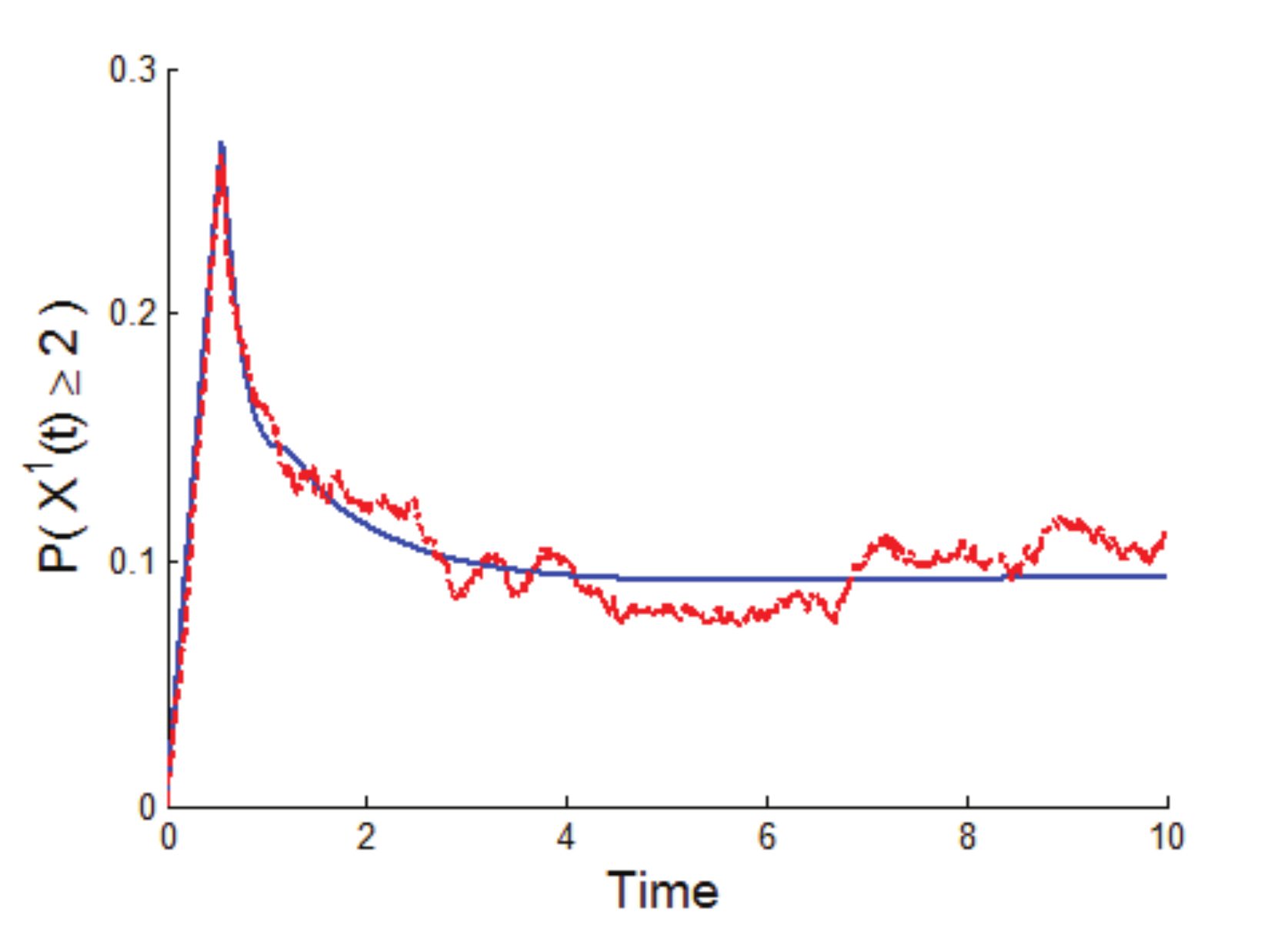}}\quad
\subfigure[Log-Normal]{\includegraphics [height =2.8 cm]{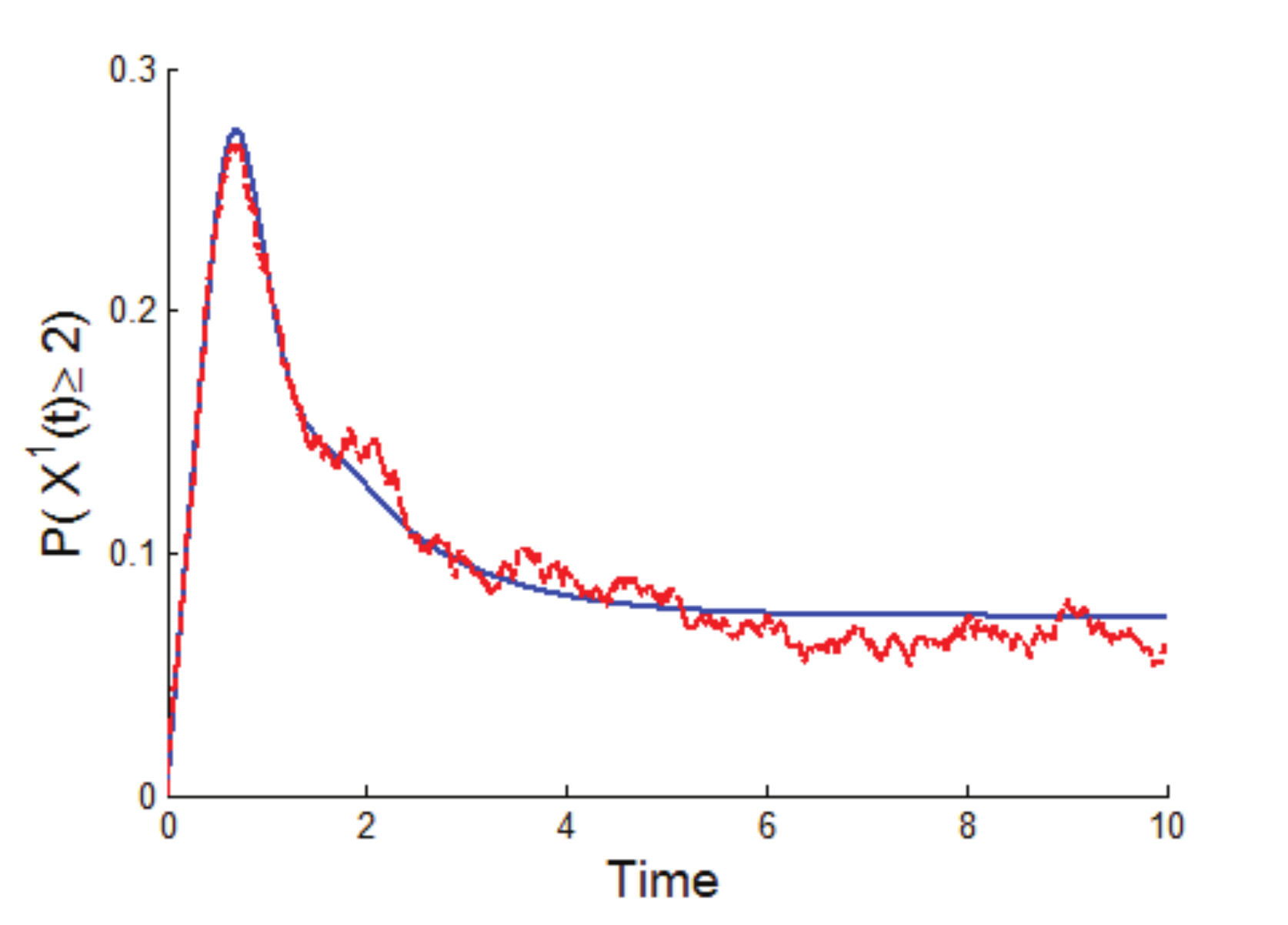}}\quad
\subfigure[Gamma]{\includegraphics[height = 2.8 cm ]{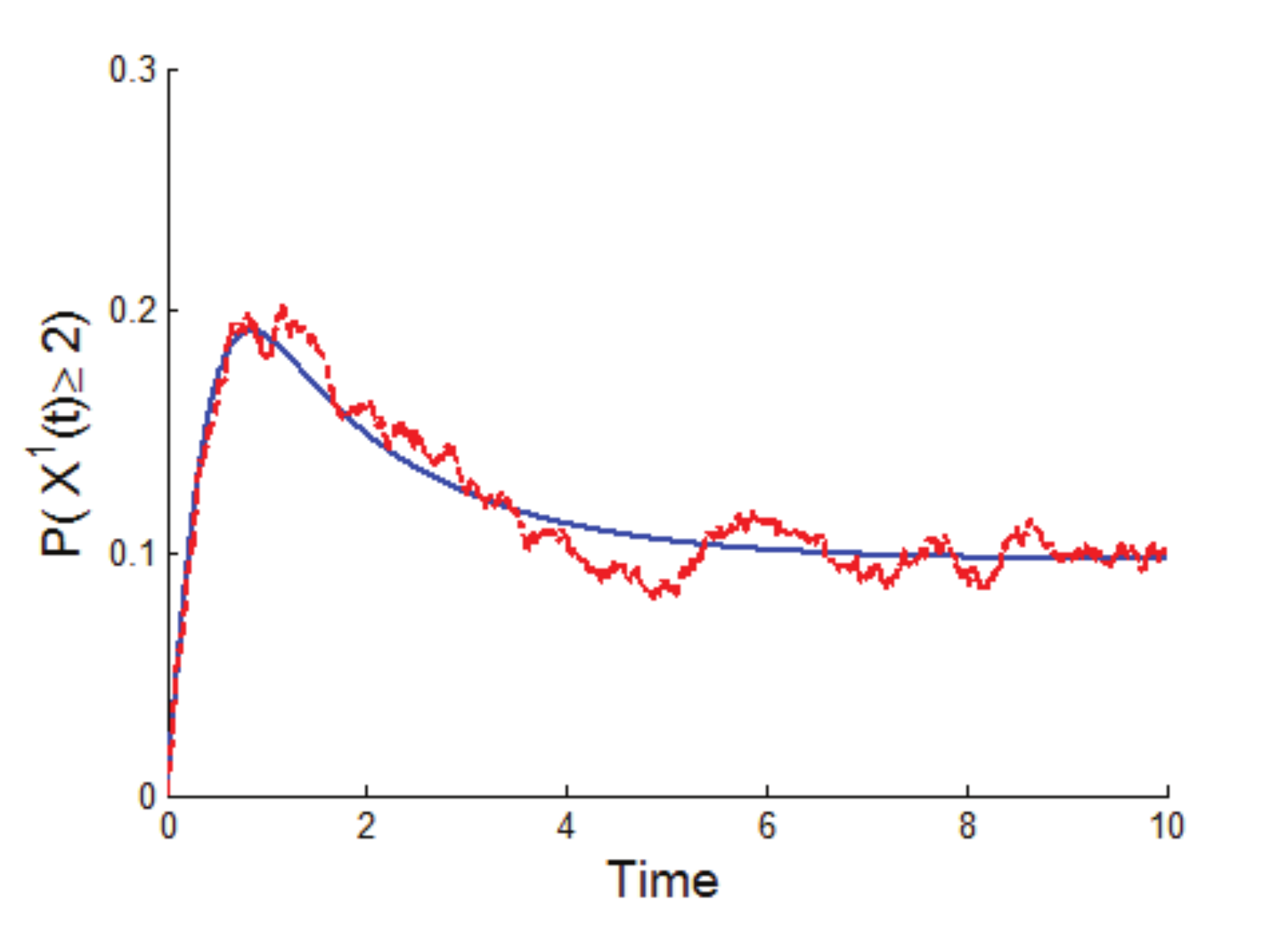}}\quad
\subfigure[Hyper-Exponential]{\includegraphics[height = 2.8 cm ]{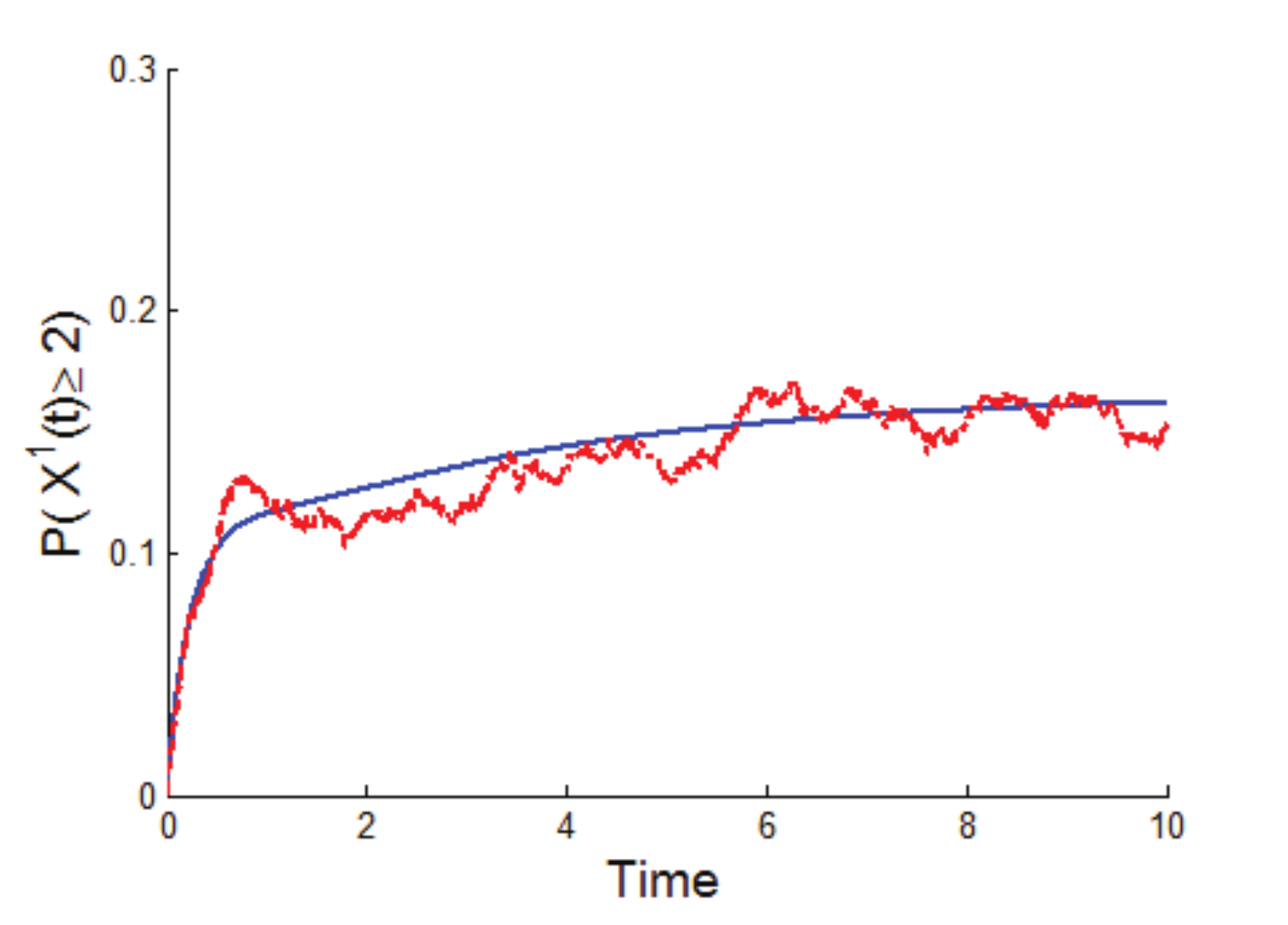}}
\end{center}
\caption{ Comparison of the estimate for $P(X^{(N),1}(t)\geq \ell)$ obtained from MC simulation (in red) versus the numerical approximation of $Z_\ell (t,0)$ (in blue) during $t \in [0,10]$ for $\ell=1$ (top row) and $\ell=2$ (bottom row).}
\label{Figure:s2}
\end{figure*}

First we compute the probability that a typical (fixed) queue has length at least $\ell$ at time $t$, for $t \in [0,10]$ in a network of $N=1000$ servers, using Monte Carlo simulation with $1000$ realizations. Then we compare this probability
with the quantity obtained by numerically solving the PDE using the method described in Section \ref{sec_simApprox}, with $L_0=6$, $R_0=20$ and $\delta=0.001$. We make this comparison for a variety of (unit mean) service time distribution, including Pareto (with shape parameter $\beta=2.25$),  Log-Normal (with shape parameter $\sigma=0.33$), Gamma (with shape parameter $k=2$) and Hyper-Exponential (with parameters $\lambda_1=0.5$ and $\lambda_2=2$), which is a special case of a Phase-type distribution.

The results are illustrated in Figure  ~\ref{Figure:s2} for $\ell=1$ and $\ell=2$, and a close match is observed in all cases. In our setting, the run time for the Monte Carlo method is approximately $2-3$ hours, which is orders of magnitude longer than the time taken to numerically approximate the PDE, which is around $7-9$ seconds.

\subsubsection{Waiting Time}
Next, we compare the mean virtual waiting time in the same setting described above, but with the arrival rate now set to $\lambda=0.7$. We measure the mean virtual waiting from a Monte Carlo simulation as follows: at any time $t$, we determine which server a virtual arriving job would have been routed to by choosing two servers uniformly at random, and picking the one with shorter queue length, and then observe the waiting time in that server. The average is taken over $2000$ realizations. We compare this to the following approximation of the limit provided by Theorem \ref{thm_wt}:
\begin{align}\label{virtual_wait}
\Ept{\waitn(t)}&\approx \sum_{\ell=2}^{L_0}\hat Z_{\ell}(t,0)^2+\sum_{\ell= 1}^{L_0-1} \left[ \hat Z_{\ell}(t,0)+\hat Z_{\ell+1}(t,0)\right]\sum_{j=0}^{\lfloor R_0/\delta\rfloor}\left[\hat Z_{\ell}(t,r_j)-\hat Z_{\ell+1}(t,r_j)\right]\delta.
\end{align}
where $\{\hat Z_\ell;\ell\geq1\}$ is the numerical solution of the PDE described in Section \ref{sec_simApprox}.

The result of the comparison for the Pareto distribution (with mean set to $1$ and shape parameter $\beta=3$) and  is plotted in Figure~\ref{Figure:waiting}. We observe good agreement between the two curves. Furthermore, recall that the actual waiting time of a job is the difference between the arrival and service entry times of that job. We also plot the average of the average of the actual waiting time.  At each time $t$, the latter quantity is defined to be the sum of the waiting times of all jobs arrived in that time slot in the mesh, divided by the number of jobs that arrived in that time slot.  We observe that the mean virtual waiting time is a good approximation to the average actual waiting time as well.

\begin{figure}
\begin{center}
\subfigure{\includegraphics[height = 5 cm ]{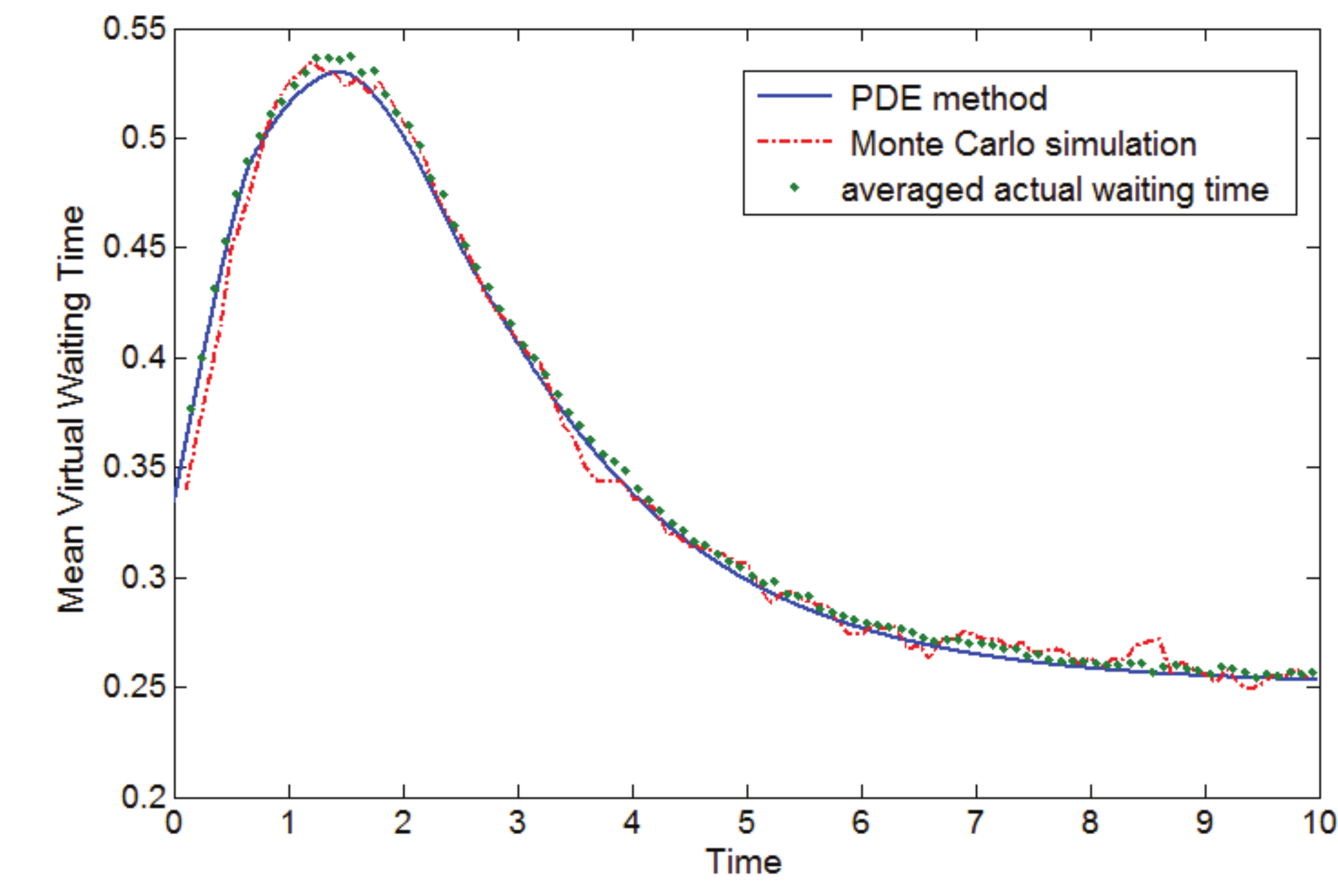}\label{Figure:waiting}}
\hspace{1.5cm}
\subfigure{\includegraphics[height = 5 cm ]{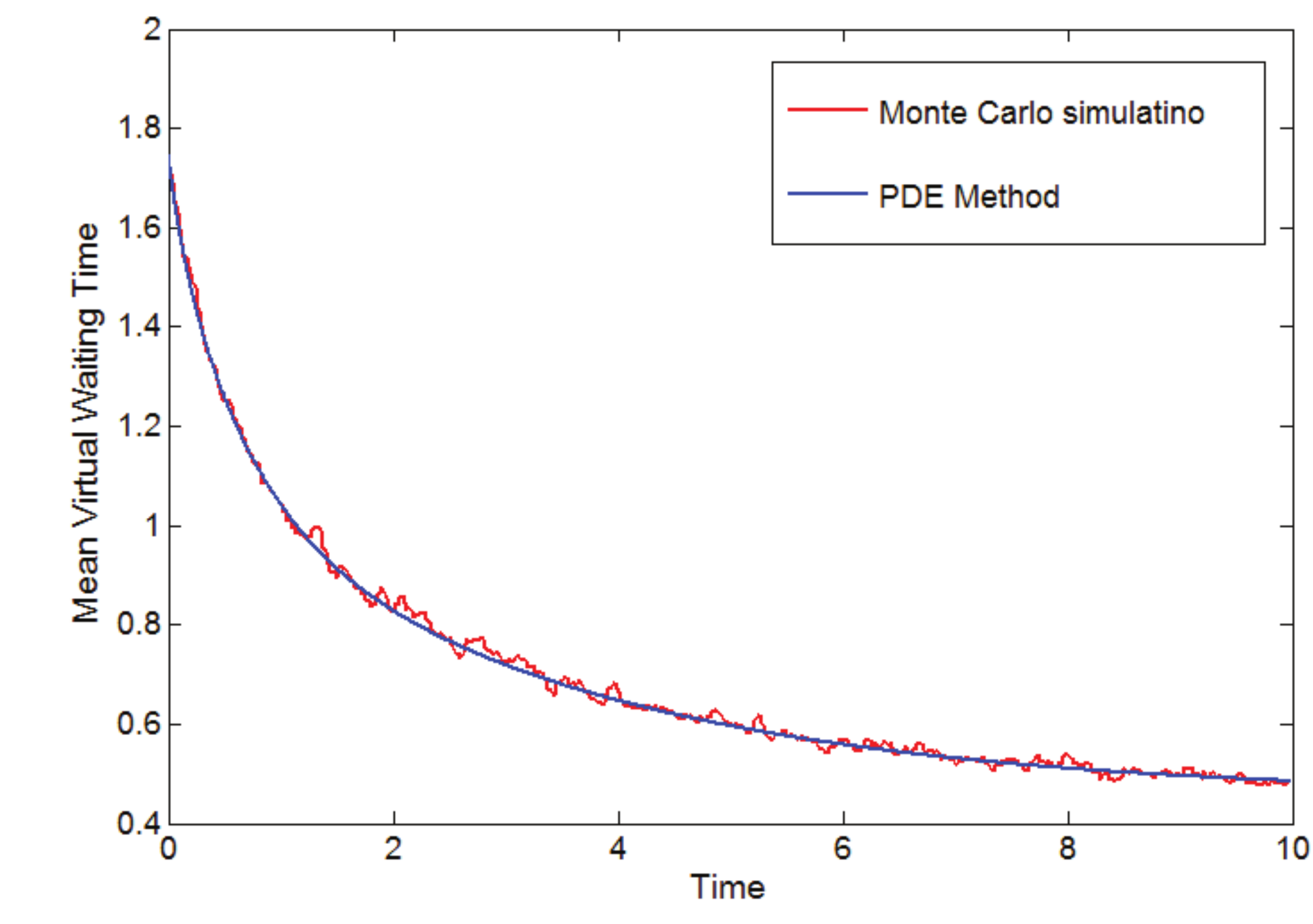}\label{Figure:waiting_gamma}}
\end{center}
\caption{\textbf{a.} Mean virtual sojourn times from MC (red) and  PDE (blue) as well as the averaged actual sojourn times from MC (green) for the Pareto distribution with $\beta =3$. \textbf{b.} Mean virtual waiting times from MC (red) and  PDE (blue) for Gamma service distribution and unbounded initial ages.}
\end{figure}

\subsubsection{More General Initial Conditions}\label{sec_moregeneral}
Finally, to illustrate that the assumption $\ageni(0)\le T_0$ on initial ages is not necessary, we validate our result for another sequence of networks with the following initial condition: there are initially two jobs in each queue, and the initial ages are all independent and distributed according to $\overline G(x)dx$ (which is the stationary age distribution in a renewal process with inter-arrival distribution $G$). The corresponding initial condition for the PDE are given by
\[Z_\ell^0(r)=\int_r^\infty\overline G(x)dx,\quad\ell=1,2,\quad\quad\quad Z_\ell^0(r)=0,\quad\ell\geq3.\]

We validate the approximation given in \eqref{virtual_wait} using Monte Carlo simulations for a network with $N=1000$ servers  and a Gamma service time distribution (with shape parameter $k=2$). The result is depicted in Figure ~\ref{Figure:waiting_gamma}.

\section{Engineering Insights}\label{sec_insight}
In practice, transient Quality of Service (QoS) parameters are of particular interest in many applications. However, to the best of our knowledge, prior to this work, load-balancing networks with general service distributions have only been studied in steady state and under constant Poisson arrivals \cite{BraLuPra13}. We illustrate how our PDE approximation can be used to shed insight into transient phenomena of practical relevance.  In Section \ref{sec_backlog}, we study the time it takes for a congested network to get rid of a backlog of jobs, and in Section \ref{sec_proiodic} we analyze the performance of a load balancing network with time-varying, periodic, Poisson arrivals, consisting of intermittent ``high'' and ``low'' periods.

\subsection{Initial Backlog}\label{sec_backlog}

Given a network that is congested with a backlog of jobs, the system administrator would like to know how long it would take for the system to get rid of the backlog and for the QoS to return closer to the normal operating point. In this section, we consider a network with a large initial backlog, and hence a large initial virtual waiting time, and study the time it takes for the system to unload to the extent that the mean virtual waiting time reaches half of its initial value, which we refer to as ``relaxation time''.  We investigate the effect of service distribution statistics on the relaxation time.  In each case, the goal is to illustrate how the PDE approximation may be used to help uncover interesting (and possibly unexpected) network phenomena.

In the presence of general service distributions, in order to capture congestion in the initial distribution of the network, one has to specify not only the distribution of the number of jobs in the different queues, but also the distribution of ages of jobs being served at these queues.   There are a number of different configurations that could reflect a congested initial state.  In order to generate initial conditions that may naturally occur in practice, we consider an initial state that corresponds to the distribution resulting from a network that has been experiencing a higher than normal arrival rate for a period of time. Specifically, we first consider a network that starts empty and runs with nominal arrival rate $\lambda = 0.6$ for $T_0=10$ units of time, and then experienced an  arrival rate that is about $8$ times the nominal arrival rate, namely $\lambda = 5$, for a period of $T_{b}=2$ units prior to zero.  The distribution of this network at time $T_0+T_{b}$ then represents a congested state for the network and is used as the initial condition for our PDE approximations.  We then assume that at time $0$, the mean arrival rate reverts back to its nominal value $\lambda = 0.6$  and we study the relaxation time, namely the time it takes for the mean virtual waiting time to reach half its initial value.

In Figure ~\ref{pareto}, we plot the evolution in time of the mean virtual waiting time for (nominal) arrival rate $\lambda = 0.6$, both when the service distribution is Pareto with unit mean and parameter $\beta = 1.25$ (heavy-tailed) and Pareto with parameter $\beta = 2.50$ (light-tailed). The curves are obtained using the numerical scheme to solve the PDE with cutoff values $L_0 = 12$ and $R_0=20$ and step size $\delta=0.001$.  An interesting observation is that the heavy-tailed case clearly has a smaller relaxation time than the light-tailed case. This phenomenon may seem particularly surprising in light of the steady-state result of Bramson et. al in \cite{BraLuPra13}, which shows that for Pareto service time distributions with unit mean, the tail of the limit steady state queue-length distribution has a double exponential decay when $\beta > 2$ (light-tailed), in contrast to the case when $\beta < 2$ (heavy-tailed), when it has only a power law decay.   Although the tails of the limit  steady-state distribution do not represent the limit of the tails of the steady-state distribution in the $N$-server system (since the $N\to\infty$ and tail decay limits are typically not interchangeable), the results in \cite{BraLuPra13} suggest that from the point of view of equilibrium queue length, the light-tailed Pareto distribution shows better performance.

The transient phenomenon observed above also persists when the initial condition is instead chosen to be a large number of jobs uniformly distributed amongst different queues, all starting with zero ages, which may roughly correspond to congestion caused due to a sudden large spurt of job arrivals into the network.

\begin{figure}
\centering
\subfigure{\includegraphics[height=5 cm]{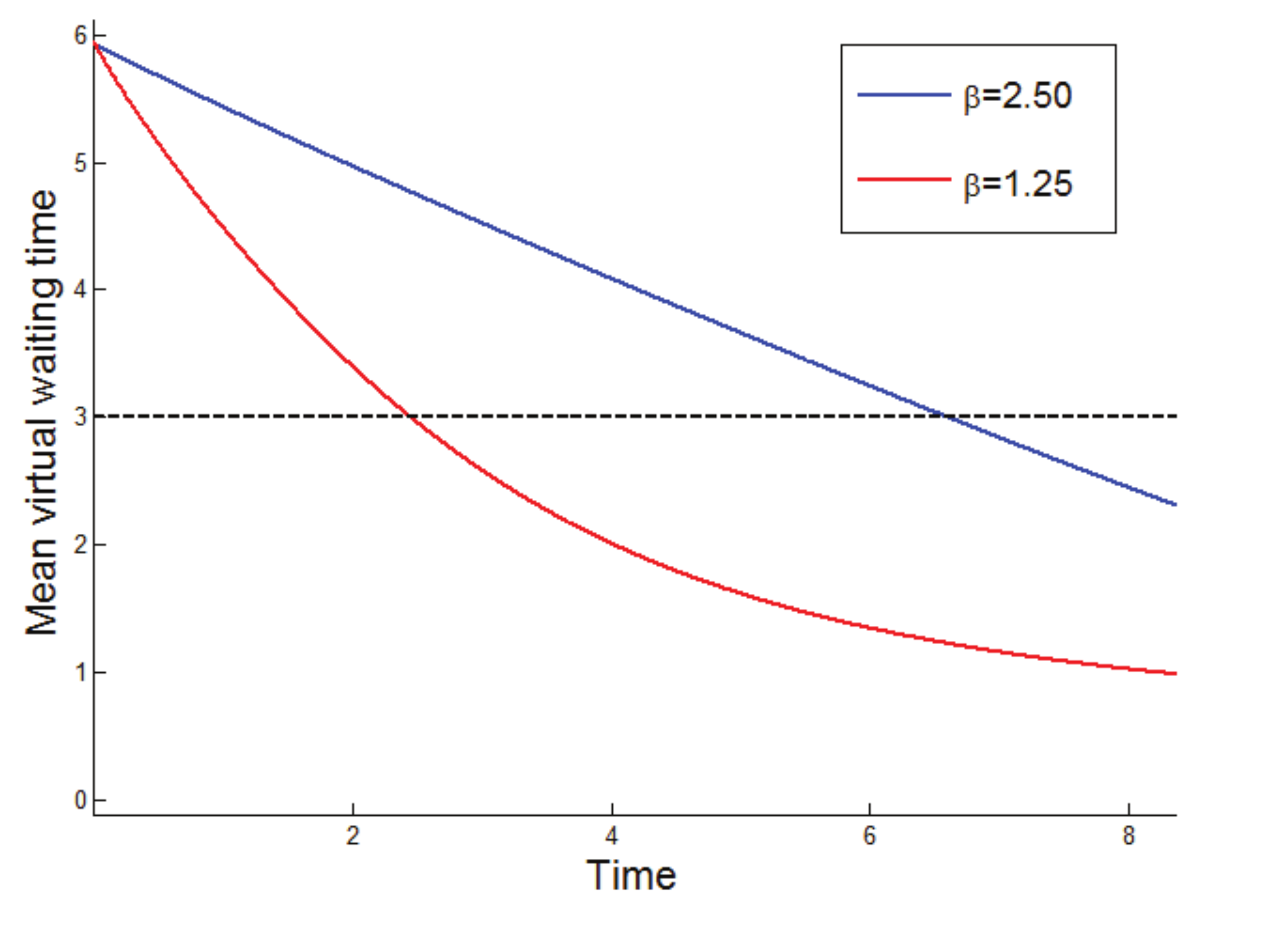}\label{pareto}}
\hspace{1.5cm}
\subfigure{\includegraphics[height=5 cm]{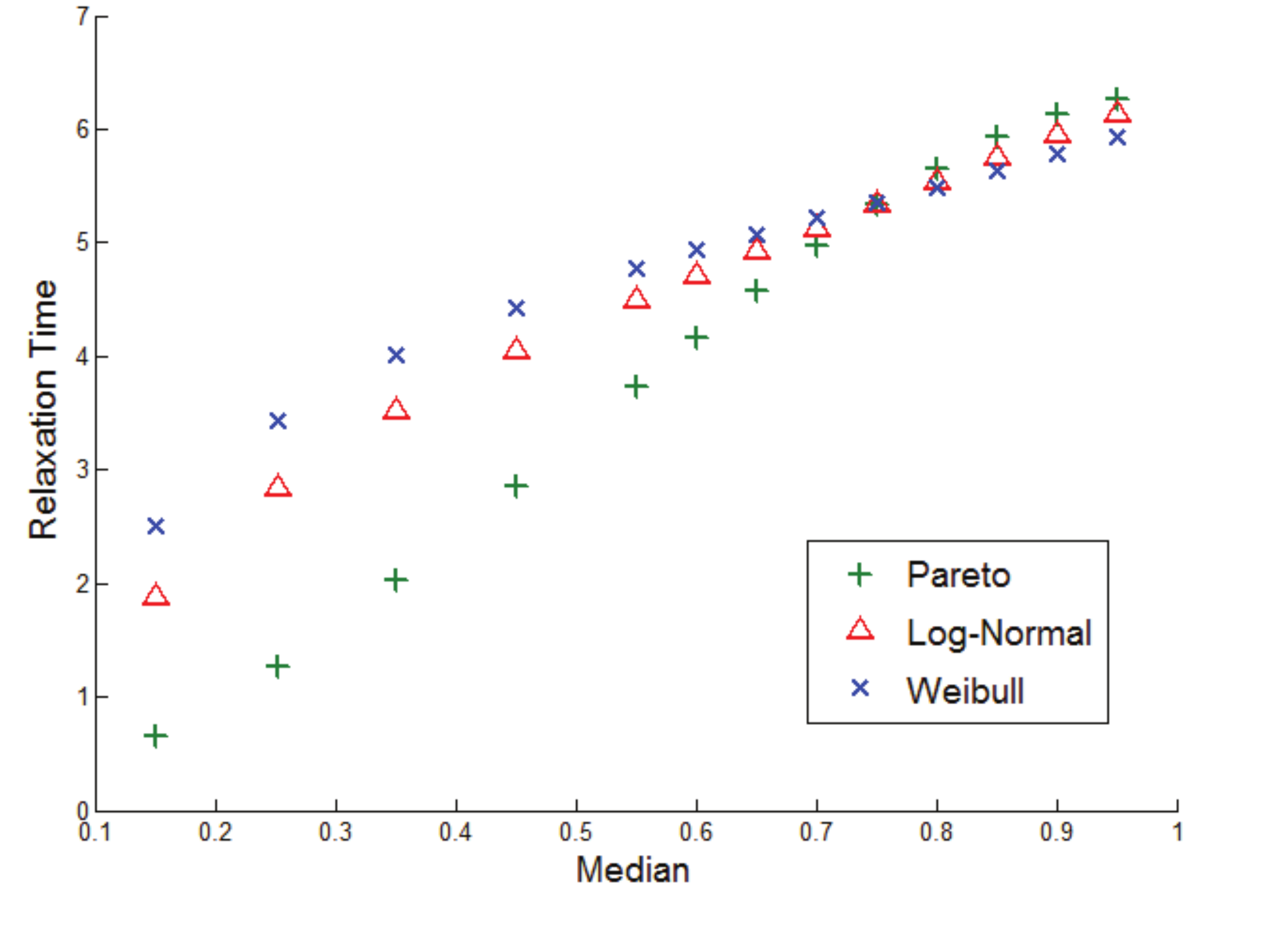}\label{trend}}
\caption{\textbf{a.}  Mean virtual waiting time for network with initial backlog and Pareto service distribution with  $\beta=1.25$ (red) and $\beta=2.50$ (blue), using the PDE method. \textbf{b.} Relaxation time vs. Median for Pareto, Log-Normal and Weibull service time distributions.}
\end{figure}

A possible heuristic explanation for the contrasting behavior observed in the transient case is that when a Pareto service distribution is fixed to have mean equal to one, the median of the distribution decreases with a decrease in $\beta$. As a result, when the heavier the tail of the service distribution, the greater the fraction of initially backlogged (and newly arriving) jobs with smaller service requirements, and the smaller the fraction of jobs with very long service times.  This helps servers in a large number of queues to get rid of their backlog and reduce their queue lengths faster, whereas the jobs with long service times lead to large queue lengths in only a small fraction of servers.  The latter jobs do not increase the mean virtual waiting time significantly because a new potential arrival will avoid the few long queues with high probability under the $SQ(2)$ load-balancing policy. This argument suggests that the same trend should hold for different classes of service time distributions.  In Figure ~\ref{trend}, we plot the relaxation times obtained via the PDE approximation for the Pareto, Weibull and Log-Normal service time distributions for different median values, and observe that the relaxation time decreases as the median decreases (and variance increases) for all these distribution families.

\subsection{Periodic Arrivals and Effective Arrival Rate}\label{sec_proiodic}

In contrast to prior work, our method also allows us to study the $SQ(d)$ network under time-varying arrivals. In many real-world applications, the arrival process is periodic (over period of a day, for example,) comprised of peak and off-peak periods. In this scenario, we examine the effect of the time- inhomogeneity of the arrival on the performance of the network.

As an illustration, we consider a family of periodic arrival rate functions $\lambda (\cdot) \in L^1_{\text{loc}}\ho$, with period $T$, parameterized by the average arrival rate $\overline{\lambda}$ over the period and a constant $\Delta > 0$, which can be viewed as a burstiness parameter.  The arrival rate takes the value $\overline{\lambda} + \Delta$ in the first-half of the period and the value $\overline{\lambda} - \Delta$ in the second half of the period, corresponding to peak and off-peak periods.   The larger the $\Delta$, the more dramatic the difference between the peak and off-peak arrival rates; the particular case $\Delta = 0$ corresponds to the case of constant arrival rate.  Using our PDE approximation, for  fixed $\overline{\lambda}$, we study the effect of the burstiness parameter $\Delta$ on the mean virtual waiting time of the network.

As one would expect, the burstiness has a negative impact on the network performance as it results in an increase in the mean virtual waiting time. However, the PDE approximation also quantifies this effect, and allows comparisons across different service time distributions.  To this end, for a fixed average arrival rate and different values of the burstiness parameter $\Delta$, we compute the mean virtual waiting time averaged over one period, denoted by $\overline W(\overline \lambda,\Delta)$, and find the \textit{constant rate} Poisson arrival that results in the same averaged virtual waiting time. We denote this constant rate by $\lambda_{\text{eff}}=\lambda_{\text{eff}}(\overline\lambda,\Delta)$, and refer to it as the ``effective arrival rate'' corresponding to $\overline\lambda$ and $\Delta$, which hence has the property
 \[\overline W(\overline\lambda,\Delta)=\overline W(\lambda_{\text{eff}},0).\]

We compare the effect of traffic burstiness on the performance of the network under a heavy tail and a light tail Pareto service distribution. In Figure ~\ref{Figure:waiting_period}, we plot $\lambda_{\text{eff}}$ as a function of $\Delta$ for $\overline\lambda=0.7$ and Pareto service time distributions with shape parameters  $\beta=1.5$ (heavy tail, infinite variance) and $\beta=3.0$ (light tail, finite variance.) As illustrated in the figure, the burstiness of the arrival rate has a greater effect on the network when the service time has finite variance ($\beta = 3.0$) rather than infinite variance ($\beta = 1.5)$. In other words, as far as the average mean virtual waiting time is concerned, the network with the heavy-tailed service time shows less increase in $\lambda_\text{eff}$ than the light-tailed service time. This is in line with our observation of backlogged networks that heavy tails seem to have a less deleterious effect on the transient behavior of the network than in equilibrium.

It is worth mentioning that quantitative insights into the network such as those illustrated in Figure ~\ref{Figure:waiting_period}, may require the solution of inverse problems, such as the computation of  $\lambda_\text{eff}$ for various values of $\Delta$, which would be incredibly time-consuming, if not infeasible using Monte Carlo simulation.

\begin{figure}
\centering
\includegraphics[height= 5.3 cm ]{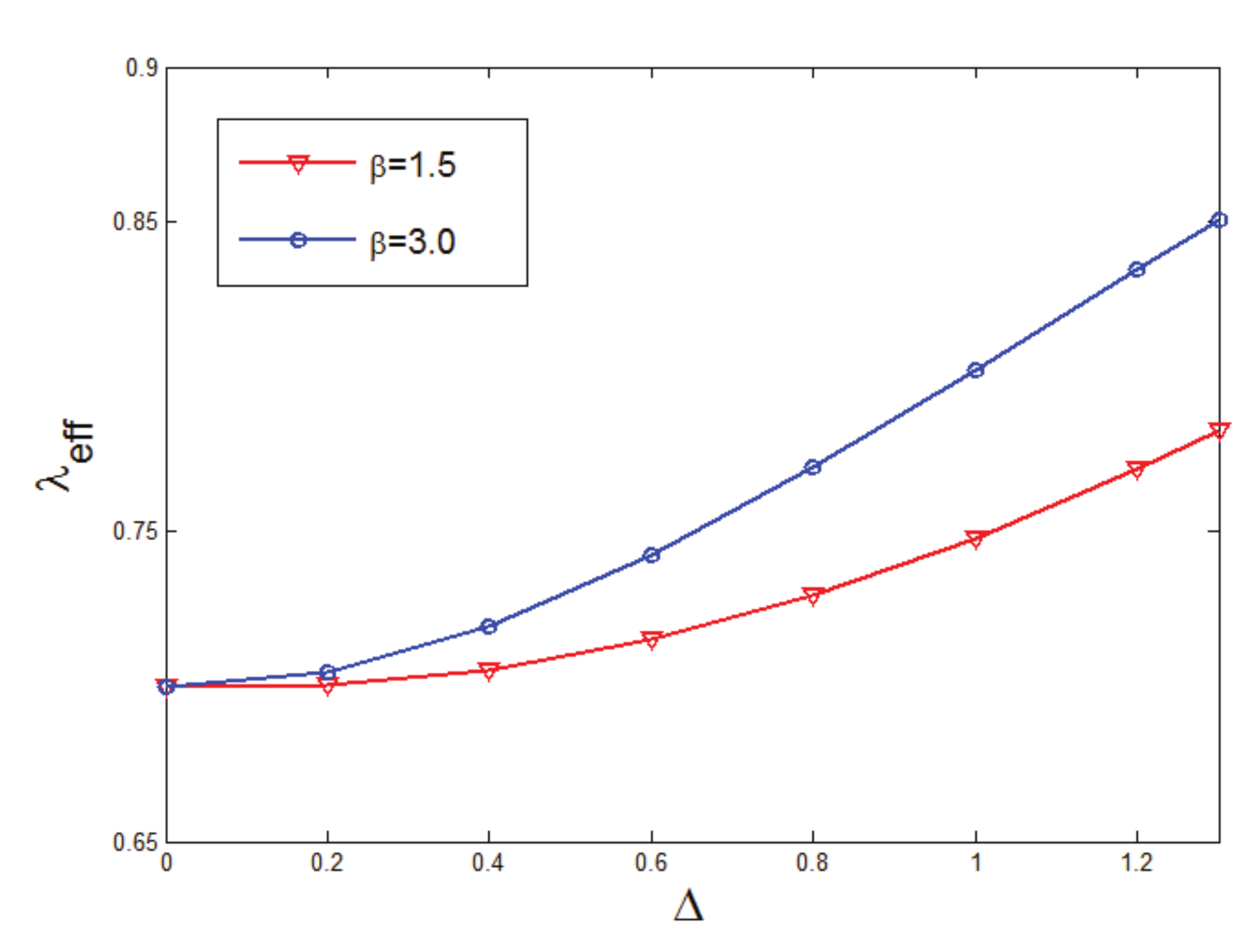}
\caption{Effective arrival rate $\boldsymbol{\lambda_{\text{eff}}}$ vs. burstiness $\Delta$ for Pareto service distributions with $\beta=1.5$ and $\beta=3$. }
\label{Figure:waiting_period}
\end{figure}

\section{Discussion and Future Work}\label{sec_discussion}
In this paper, we have introduced a new framework to analyze load-balancing networks that use the $SQ(d)$ algorithm, focusing for simplicity on the $SQ(2)$ algorithm.  We introduced the hydrodynamic PDE, which captures the evolution of the scaled state of the network in the limit as the number of servers $N$ tends to infinity.  We proved that the PDE has a unique solution and proposed a numerical scheme for efficiently solving the PDE, thus providing a more computationally effective method than Monte Carlo simulations for analyzing properties of large networks.  As an illustration, we applied our approximation to study transient performance measures such as the relaxation time in a backlogged network and the effect of traffic burstiness on a network with periodic time-varying arrivals.

There are many avenues for future research to extend this work. Firstly, it would be worthwhile to continue to study the effect of heavy tails on transient measures of network performance under different scenarios and to also investigate and rigorously establish convergence of the numerical scheme, as the truncation parameters $L_0$, $R_0$ go to infinity and the mesh size $\delta$  goes to zero, and obtain convergence rates.  It would also be of interest to directly study the PDE to analytically establish properties of the limit dynamics.  Secondly, as in the ODE method for exponential service time distributions, we would also like to study the equilibrium properties of the PDE.  In future work, we hope to show that the PDE has a unique fixed point and characterize its dependence on network parameters.

More broadly, our state representation and overall framework can be applied, with suitable small modifications, to analyze other load-balancing algorithms in the presence of general service distributions, both in the transient and steady state regimes. We hope to use the insight gained from such analyses to also design new algorithms (in both the homogeneous setting considered here, as well as heterogeneous settings) that may lead to better performance.

\bibliographystyle{abbrv}
\bibliography{lbref}

\end{document}